\documentclass[11pt]{article}
\usepackage{sychung2024}
\renewcommand\arraystretch{1.5}
\usepackage{subcaption}

\title{Complex zeros of Bessel function derivatives and associated orthogonal polynomials}

\author{Seok-Young Chung\footnote{seok-young.chung@ucf.edu. Department of Mathematics, University of Central Florida,
4393 Andromeda Loop N., Orlando, FL 32816, USA.}
\and Sujin Lee\footnote{watertrue015@gmail.com. Department of Mathematics, College of Natural Sciences, Chung-Ang University,
84 Heukseok-Ro, Dongjak-Gu, Seoul 06974, Korea.}
\and Young Woong Park\footnote{ywpark1839@gmail.com. Department of Mathematics, College of Natural Sciences, Chung-Ang University,
84 Heukseok-Ro, Dongjak-Gu, Seoul 06974, Korea.}}
\date{}

\begin{document}

\maketitle

{\bf Abstract.} 
We introduce a sequence of orthogonal polynomials whose associated moments are the Rayleigh-type sums, involving the zeros of the Bessel derivative $J_\nu'$ of order $\nu$. We also discuss the fundamental properties of those polynomials such as recurrence, orthogonality, etc.
Consequently, we obtain a formula for the Hankel determinant, elements of which are chosen as the aforementioned Rayleigh-type sums.
As an application, we complete the Hurwitz-type theorem for $J_\nu'$, which deals with the number of complex zeros of $J_\nu'$ depending on the range of $\nu$.

\medskip


\medskip


\section{Introduction}
The first kind Bessel function of order $\nu$ is given by
\begin{equation*}
    J_\nu(x) = \sum_{k=0}^\infty \frac{(-1)^k}{k!\, \Gamma(\nu+k+1)} \rb{ \frac{x}{2} }^{2k+\nu},
\end{equation*}
where $\Gamma(x)$ denotes the gamma function. 

Over the past few decades, numerous authors (see for instance \cite{Sneddon}, \cite[\S 15.51]{Watson}, \cite{Giusti}, \cite{Kostin}, \cite{Frantzis} and references given there) have investigated the infinite series, involving the zeros of $J_\nu$ and $J_\nu'$ respectively, given by
\begin{equation}\label{RL1}
    \sigma_{\nu}(m) = \sum_{k\in \mathbb{Z}\setminus\cb{0}} \frac{1}{\big(j_{\nu,k}\big)^{m}},\quad \sigma_{\nu}'(m) = \sum_{k\in \mathbb{Z}\setminus\cb{0}} \frac{1}{\big(j_{\nu,k}'\big)^{m}},
\end{equation}
for $m=2,3,\cdots$, where $\big\{ j_{\nu,k} \big\}_{k=1}^\infty$ and $\big\{ j_{\nu,k}' \big\}_{k=1}^\infty$ denote the sequences of positive zeros of $J_{\nu}$ and $J_{\nu}'$, respectively, arranged in ascending order of magnitude. Note that $j_{\nu,-k}= -j_{\nu,k}$ represents the $k$th negative zero as the zeros of $J_\nu$ have a symmetry around the origin. Similarly $j_{\nu,-k}'= -j_{\nu,k}'$ for $k\in \mathbb{N}$. As long as $J_\nu$ and $J_\nu'$ have real zeros, which are satisfied when $\nu > -1$ and $\nu > 0$ respectively, the infinite series in equation \eqref{RL1} are referred to as Rayleigh-type sums.

The Lommel polynomials emerge as the orthogonal polynomials determined by moments of order $n$, as $\mu_n = \sigma_\nu(n+2)$ (see \cite[pp. 367--369]{Grosjean} for details).
We now shift our attention to the case where $\mu_n = \sigma'_\nu(n+2)$, $n=\mathbb{N}\cup\cb{0}$. A monic sequence of associated orthogonal polynomials is given by
\renewcommand{\arraystretch}{1}
\begin{equation}\label{P1}
    P_n(x) = \frac{1}{\Delta_{n-1}} \dm{ \mu_0 & \mu_1 & \cdots & \mu_n \\
    \mu_1 & \mu_2 & \cdots & \mu_{n+1}\\
    \vdots & \vdots & \ddots & \vdots \\
    \mu_{n-1} & \mu_n & \cdots & \mu_{2n-1}\\
    1 & x & \cdots & x^n},
\end{equation}
where $\Delta_m = \det(\mu_{i+j-2})_{1\le i,j\le m+1}$ and $\Delta_{-1}=1$. Moreover, since $\mu_n=0$ for odd $n$, it satisfies the three-term recurrence of the form
\begin{equation}\label{P2}
    P_{n+1}(x) = x P_{n}(x) - \frac{\Delta_{n-2}\Delta_n}{\Delta_{n-1}^2} P_{n-1}(x),
\end{equation}
provided that $\Delta_m \ne 0$, $m\in\mathbb{N}\cup\cb{0}$.
The above constructions follow \cite[(2.1.4) and (2.1.6)]{Ismail1} and \cite[Theorem 4.2 and 4.3]{Chihara}. Although $\Delta_n$ is of great importance to the expressions in \eqref{P1} and \eqref{P2}, the explicit form of the Hankel determinant $\Delta_n$ has not yet been addressed. In sections 2 and 3, we shall introduce the sequences of associated polynomials with $J_\nu'$ along with their elementary properties and derive an explicit formula of $\Delta_n$ for each $n\in\mathbb{N}$.

As for the application of $\Delta_m$ with its explicit formula, we have applied \cite[Theorem 7 and 9]{Kytmanov}, which states as follows:
\begin{varthm}{A}\label{thm:A}
    Let $f$ be a real entire function of order of growth $1$ with an infinite number of zeros $\{\zeta_n\}_{n=1}^\infty$, repeated according to multiplicity.
    Define
    \begin{equation*}
        \sigma_k=\sum_{n=1}^\infty \frac{1}{\zeta_n^{k}},\quad k\ge 2,
    \end{equation*}
    and
    \begin{equation*}
        \mathcal{D}_n = \begin{vmatrix}
            \sigma_{2} & \sigma_{3} & \cdots & \sigma_{n+2} \\
            \sigma_{3} & \sigma_{4} & \cdots & \sigma_{n+3}\\
            \vdots & \vdots & \ddots & \vdots\\
            \sigma_{n+2} & \sigma_{n+3} & \cdots & \sigma_{2n+2}
        \end{vmatrix},\quad n=0,1,\cdots,
    \end{equation*}
    with $\mathcal{D}_{-1}=1$. If the sequence $\{\mathcal{D}_{n-1} \mathcal{D}_{n}\}_{n=0}^\infty$ contains $m$ negative numbers, then $f$ has exactly $m$ conjugate pairs of complex zeros and an infinite number of real zeros.
\end{varthm}
\noindent
This approach has been introduced by Baricz and {\v S}tampach \cite{Baricz1} to determine the number of complex zeros for the regular Coulomb wave functions for given parameters. Their result extends the classical theorem for the Bessel functions due to Hurwitz \cite{Hurwitz}, which states
\begin{varthm}{B}[Hurwitz]\label{thm:B}\
    \begin{enumerate}[label=\rm(\roman*)]
        \item If $\nu>-1$, then $J_\nu$ has only real zeros.
        \item If $-k-1 < \nu <-k$, where $k$ is positive odd integer, then $J_\nu$ has exactly $2k$ complex zeros among which are $2$ purely imaginary zeros.
        \item If $-k-1< \nu <-k$, where $k$ is positive even integer, then $J_\nu$ has exactly $2k$ complex zeros, none of which are purely imaginary.
    \end{enumerate}
\end{varthm}
\noindent It has attracted the interest of numerous authors. We refer to \cite{Baricz1} and references therein. In this paper, it is important to note that the term `complex zero' is consistently used to refer to `nonreal zero' in all contexts.

An analogue of Hurwitz theorem for $J_\nu'$ remains incomplete, although partial results have been discussed. 
The relevant historical results are summarized below, presented in chronological order:

\begin{varthm}{C}\label{thm:C}\
    \begin{enumerate}[label=\rm(\roman*)]
    \item {\rm(Lommel \cite[\S 6]{Olver})} $J_\nu'$ has only real zeros for $\nu\ge0$.

    \item {\rm(Dixon \cite{Dixson}, \cite[\S 15.25]{Watson})} For $-1<\nu<0$, all zeros of $J_\nu'$ are real except for a single conjugate pair of purely imaginary zeros. \

    \item {\rm(Lense \cite[pp. 149-151]{Lense})} 
    For $\nu=-1,\,-2,\,-3,\cdots$, the zeros of $J_\nu'$ are all real. Furthermore, for sufficiently small $\epsilon>0$ and $k\in \mathbb{N}$, when $-k-\epsilon<\nu<-k$, $J_\nu'$ has $2$ real and $2k-2$ complex zeros near the origin, while when $-k<\nu<-k+\epsilon$, $J_\nu'$ has $2k$ complex zeros near the origin. Besides, $J_\nu'$ has $2$ complex zeros near the origin for $-\epsilon<\nu <0$ while $J_\nu'$ has $2$ real zeros near the origin for $0<\nu<\epsilon$.

    \item {\rm(Lorch and Muldoon \cite[Theorem 5.1]{Lorch2})} There exists a unique $\nu^*$ within the interval $(-2,-1)$ such that $J_\nu'$ has exclusively real zeros when $\nu^*<\nu<-1$. 
    \end{enumerate}
\end{varthm}

In addition, Ismail and Muldoon \cite{Ismail4} proved that the Dini function $\alpha J_\nu(x) + \beta x J_\nu'(x)$ ($\beta\ne0$) has a single conjugate pair of purely imaginary zeros for $-1<\nu < -\alpha/\beta$, but only real zeros when $\nu \ge -\alpha/\beta$. It was also found by Baricz et al. \cite{Baricz2} that for $n\ge1$ and $\nu >n-1$, all zeros of $J_\nu^{(n)}(x)$ are real. 

Section 4 completes this Hurwitz-type theorem for $J_\nu'$ by providing a detailed description of the number of complex zeros of $J_\nu'$ depending on the range of a parameter $\nu$. In section 5, we provide a detailed proof of Theorem \ref{thm:main}.

\section{Orthogonal polynomials \texorpdfstring{$q_{n,\nu}$}{q n nu} originating from the continued fraction}

The correlation between continued fractions and orthogonal polynomials is widely studied in many literature works. This connection dates back to as early as 1894, when orthogonal polynomials were introduced within the framework of continued fractions by Stieltjes \cite{Stieltjes}. Our goal in this section is to derive a sequence of orthogonal polynomials using the context of continued fraction. Furthermore, the key properties of these polynomials, such as recurrence relations and orthogonality, will be discussed.

A notable distinction from the case of $J_\nu$ (as demonstrated in \cite{Grosjean}) lies in the absence of a three-term recurrence relation for $J_\nu'$ with respect to $\nu$. However, the well-known identities
\begin{align}
    J_\nu'(x) &= \frac{\nu}{x} J_\nu(x)  - J_{\nu+1}(x), \notag \\
    J_{\nu-1}(x) &= \frac{2\nu}{x} J_\nu(x)  - J_{\nu+1}(x), \label{Id}
\end{align}
serve to induce the continued fraction
\begin{equation*}
    \frac{J_\nu(1/x)}{J_\nu'(1/x)} =  \frac{ 1 }{ \displaystyle \nu x - \frac{\hspace{1.5em}1}{\displaystyle 2(\nu+1)x - \frac{1}{ \displaystyle 2(\nu+2)x - \frac{1\hspace{1.5em}}{ \displaystyle 2(\nu+3)x - \cdots   }}}},
\end{equation*}
or equivalently,
\begin{equation*}
    \frac{J_\nu(1/x)}{J_\nu'(1/x)} = \frac{1 \vert}{ \vert \nu x }- \frac{1 \vert}{\vert 2(\nu+1)x } - \frac{1 \vert}{\vert 2(\nu+2)x }- \frac{1 \vert}{\vert 2(\nu+3)x } -\cdots. 
\end{equation*}

The $n$th approximant of this continued fraction inherently evokes a sequence of orthogonal polynomials, as denominator polynomials, intricately linked to $J_\nu'$. 
To be more precise, the $n$th approximant of the above is given by
\begin{equation*}
    \frac{N_{n}(x)}{D_{n}(x)} = \frac{1 \vert}{ \vert \nu x }- \frac{1 \vert}{\vert 2(\nu+1)x } - \frac{1 \vert}{\vert 2(\nu+2)x }- \cdots- \frac{1 \vert}{\vert 2(\nu+n-1)x},
\end{equation*}
where $N_n$ and $D_n$ denote respectively the numerator and denominator polynomials, which are relatively prime. Moreover, those satisfy the three-term recurrence relation (see \cite[Theorem 2.6.1]{Ismail1}), given by
\begin{equation}\label{Lommel Recur}
    y_{n+1,\nu}(x) = xy_{n,\nu}(x) - \frac{1}{4(\nu+n-1)(\nu+n)}y_{n-1,\nu}(x),
\end{equation}
for $\nu\in \R\setminus\cb{0,-1,-2,\cdots}$ with initial values
\begin{align*}
    D_0(x) &= 1,\quad D_1(x) = \frac{x}{2},\quad
    N_0(x) = 0,\quad N_1(x) = \frac{1}{2\nu}.
\end{align*}
According to literature \cite[\S 6.5]{Ismail1}, \cite[\S 9.6]{Watson}, the recurrence \eqref{Lommel Recur} aligns with the recurrence for the Lommel polynomials $\cb{R_{n,\nu}(x)}_{n=0}^\infty$, written as
\begin{equation*}
    R_{n+1,\nu}(x) = \frac{2(\nu+n)}{x} R_{n,\nu}(x)- R_{n-1,\nu}(x),
\end{equation*}
with $R_{0,\nu}(x)=1$ and $R_{1,\nu}(x) = 2\nu/x$.
Consequently, it is simple to deduce that
\begin{align*}
    N_n(x) &= \frac{1}{2^n (\nu)_n} R_{n-1,\nu+1}(1/x),\\
    D_n(x) &= \frac{1}{2^{n+1} (\nu)_n}\big(R_{n,\nu}(1/x) - R_{n-2,\nu+2}(1/x) \big),
\end{align*}
where $(a)_n$ stands for the Pochhammer symbol, given by $(a)_n = \Gamma(a+n)/\Gamma(a)$, $n=0,1,\cdots$.
We define a sequence of polynomials $\cb{q_{n,\nu}(x)}_{n=0}^\infty$ as
\begin{equation*}
    q_{n,\nu}(x) = D_n(x) = \frac{1}{2^{n+1} (\nu)_n}\big(R_{n,\nu}(1/x) - R_{n-2,\nu+2}(1/x) \big).
\end{equation*}
We recall a classical result due to Hurwitz \cite[\S 9.65 (1)]{Watson}, stated as
\begin{equation}\label{limit1}
    \lim_{n\to \infty} \frac{(x/2)^{\nu+n-1}R_{n,\nu}(x)}{\Gamma(\nu+n)} = J_{\nu-1}(x),
\end{equation}
leading to the limit of $q_{n,\nu}$ as
\begin{align}\label{limit2}
  \lim_{n\to \infty} &\frac{(x/2)^{\nu+n-1} q_{n,\nu}(1/x)}{\Gamma(\nu+n)} = \frac{J_{\nu-1}(x)-J_{\nu+1}(x)}{2^{n+1}(\nu)_n} =  \frac{1}{2^{n}(\nu)_n} J_{\nu}'(x).
\end{align}

\begin{remark}
A sequence of polynomials $\cb{q_{n,\nu}(x)}_{n=0}^\infty$ can also be derived iteratively by applying \eqref{Id} into $J_\nu'(x) = \big(J_{\nu-1}(x)-J_{\nu+1}(x)\big)/2$. Consequently,
\begin{equation*}
        J_\nu'(x) = -2^{m-1}(\nu)_{m-1}q_{m-1,\nu}(x) J_{\nu+m}(x) + 2^m(\nu)_m q_{m,\nu}(x) J_{\nu+m-1}(x).
\end{equation*}
Furthermore, for $\nu>0$, the positive zeros of $q_{m,\nu}$ compensate the positive zeros of $J_{\nu+m}$ to be interlaced with the positive zeros of $J_\nu'$, as long as $J_{\nu+m}$ and $J_\nu'$ do not share zeros (see \cite[Theorem 5.1]{CLP}).
\end{remark}

We now present the Christoffel-Darboux identities for the sequence of polynomials $\cb{q_{n,\nu}(x)}_{n=0}^\infty$ as follows:
\begin{theorem}\label{thm:2.1}
Let $\nu\in \R\setminus\cb{0,-1,-2,\cdots}$. The Christoffel-Darboux identities holds for $n\ge 0${\rm :}
\begin{align}
    \sum_{k=0}^n  \epsilon_k & \frac{q_{k,\nu}(x)q_{k,\nu}(y)}{\lambda_k}
    = \frac{q_{n+1,\nu}(x)q_{n,\nu}(y)- q_{n+1,\nu}(y)q_{n,\nu}(x)}{\lambda_{n}(x-y)}, \label{CD1}\\
    \sum_{k=0}^n  \epsilon_k & \frac{\big(q_{k,\nu}(x)\big)^2}{\lambda_k} 
    = \frac{q_{n+1,\nu}'(x)q_{n,\nu}(x)- q_{n+1,\nu}(x)q_{n,\nu}'(x)}{\lambda_{n}}, \label{CD2}
\end{align}
where $\epsilon_0=1/2$, $\epsilon_j =1$ for $j\ge 1$, and $\lambda_n = 1/[ 4^n (\nu)_n (\nu+1)_n ]$ for $n=0,1,\cdots$.
\end{theorem}

\begin{proof}
The case of $n=0$ is trivial, as it follows directly from $q_{0,\nu}(x)=1$ and $q_{1,\nu}(x)=x/2$. For $n\ge1$, we introduce the quantity $\Delta_n(x,y)= q_n(x)q_{n-1}(y)- q_n(y)q_{n-1}(x)$. Employing the recurrence \eqref{Lommel Recur}, given by
\begin{align*}
    x q_{n,\nu}(x) = q_{n+1,\nu}(x) + \beta_{n} q_{n-1,\nu}(x),
    \quad \beta_n = \frac{1}{4(\nu+n-1)(\nu+n)},
\end{align*}
and upon multiplication with $q_n(y)$ followed by subtraction with the roles of $x$ and $y$ interchanged, we obtain that
\begin{equation*}
    (x-y) q_{n,\nu}(x)q_{n,\nu}(y) = \Delta_{n+1}(x,y)-\beta_n \Delta_n(x,y).
\end{equation*}
Dividing both sides by $\lambda_n = \beta_1\cdots \beta_{n}$ and $x-y$, an iterative procedure gives
\begin{align*}
    \frac{\Delta_{n+1}(x,y)}{\lambda_{n}(x-y)} = \sum_{k=1}^n \frac{q_{k,\nu}(x)q_{k,\nu}(y)}{\lambda_k} + \frac{\Delta_{1}(x,y)}{x-y}.
\end{align*}
Since $q_{0,\nu}(x)= 1$ and $q_{1,\nu}(x)=x/2$, we can express $\Delta_1(x,y)$ as $(x-y)q_{0,\nu}(x)q_{0,\nu}(y)/2$ and we set $\lambda_0=1$, which directly demonstrates \eqref{CD1}. The second identity \eqref{CD2} follows from the limiting case $y\to x$.
\end{proof}

Following \cite[Theorem 2.2.3]{Ismail1}, it is established that the zeros of $q_{n,\nu}(x)$ are real and simple for $\nu>0$. Moreover, for $\nu>0$, the zeros of $q_{n,\nu}(x)$ and $q_{n+1,\nu}(x)$ are interlaced in the sense that
\begin{equation*}
    x_{n+1,\nu,1} < x_{n,\nu,1} < x_{n+1,\nu,2} < x_{n,\nu,2} <\cdots < x_{n,\nu,n} < x_{n+1,\nu,n+1},
\end{equation*}
where $\cb{x_{m,\nu,j}}_{j=1}^m$ represents the sequence of zeros of $q_{m,\nu}(x)$ arranged such that $x_{m,\nu,1} <\cdots< x_{m,\nu,m}$. In what follows, for a fixed $\nu$, we abbreviate $x_{m,\nu,j}$ by $x_{m,j}$, denoting the $j$th zero of $q_{m,\nu}$ in increasing order.

Replacing $n$ by $n-1$, and substituting $x=x_{n,j}$ and $y=x_{n,s}$, a direct consequence of Theorem \ref{thm:2.1} can be expressed as follows:

\begin{theorem}\label{thm:2.2}
Let $\nu\in \R\setminus\cb{0,-1,-2,\cdots}$ and $n\in\mathbb{N}$. Then we have
\begin{equation*}
    \rho(x_{n,j}) \sum_{k=0}^{n-1} \epsilon_k \frac{q_{k,\nu}(x_{n,j})q_{k,\nu}(x_{n,s})}{\lambda_k} = \delta_{j,s},
\end{equation*}
where $\delta_{j,s}$ denotes the Kronecker delta and
\begin{equation}\label{rho}
    \rho(x_{n,j}) = \frac{\lambda_{n-1}}{q_{n,\nu}'(x_{n,j})q_{n-1,\nu}(x_{n,j})}.
\end{equation}
\end{theorem}

Note that if $\nu>0$, then $\lambda_n = \beta_1\cdots \beta_n>0$ for $n\ge1$, and $\lambda_0=1>0$. Consequently $\rho(x_{n,j})$, as defined in \eqref{rho}, remains positive when $\nu>0$.

\begin{corollary}\label{coro:2.1}
Let $\nu>0$. The following holds true
\begin{align*}
    \frac{\epsilon_{j}}{\lambda_{j}}  \sum_{k=1}^{n} \rho(x_{n,k}) 
    q_{j,\nu}(x_{n,k})q_{s,\nu}(x_{n,k})= \delta_{j,s},
\end{align*}
for each $j,s=0,1,\cdots,n-1$.   
\end{corollary}

\begin{proof}
We introduce
\begin{align*}
    U = \qb{ u_{i,j} }_{1\le i,j\le n}\in \R^{n\times n},\quad u_{i,j} = \sqrt{ \frac{\epsilon_{j-1}\rho\rb{x_{n,i}}}{\lambda_{j-1}} }  q_{j-1,\nu}(x_{n,i}),
\end{align*}
provided that $\nu>0$. Theorem \ref{thm:2.2} reads that $\delta_{j,s} = \rb{U U^T}_{j,s}$.
Since $U$ is orthogonal, we have $U^TU = I$, i.e.,
\begin{align*}
    \delta_{j,s}
    = \sum_{k=1}^{n} u_{k,j} u_{k,s} = \frac{\epsilon_{j-1}}{\lambda_{j-1}}  \sum_{k=1}^{n} \rho(x_{n,k})
     q_{j-1,\nu}(x_{n,k})q_{s-1,\nu}(x_{n,k}),
\end{align*}
leading to the result by shifting $j-1 \mapsto j$ and $s-1 \mapsto s$.
\end{proof}

We now restrict ourselves to the case of $\nu>0$ throughout this section. Let us define a sequence of right continuous step functions $\cb{\psi_n(x)}_{n\ge1}$ such that $\psi_n(x+0)-\psi_n(x-0)=0$ for $x \ne x_{n,k}$, $1\le k\le n$, and
\begin{equation*}
    \psi_n(-\infty)=0,\quad \psi_n(x_{n,k}+0)-\psi_n(x_{n,k}-0) =\rho(x_{n,k}).
\end{equation*}
Then it is evident from Corollary \ref{coro:2.1} that for $0\le j,s\le n-1$,
\begin{equation*}
    \int_\R q_{j,\nu}(t)q_{s,\nu}(t) d\psi_n(t) = \frac{\lambda_j}{\epsilon_j} \delta_{j,s}.
\end{equation*}

\begin{remark}\label{remark:2.1}
We list several routine consequences based on the reference \cite[Chap. 2]{Ismail1} as follows:
\begin{enumerate}[label=(\roman*)]
    \item Since $\cb{\beta_n}_{n=1}^\infty$ is bounded for $\nu>0$, there exists the unique orthogonality measure $\mu$ with bounded support such that
    \begin{equation}\label{ortho}
        \int_\R q_{j,\nu}(t)q_{s,\nu}(t) d\mu(t) =  \frac{\lambda_j}{\epsilon_j} \delta_{j,s},
    \end{equation}
    for all $j,s=0,1,\cdots$.

    \item Helly's selection principle (see also \cite[p. 13]{Shohat}) produces a subsequence $\{\psi_{n_j}\}_{j=1}^\infty$ of a sequence of step functions $\cb{\psi_n}_{n=1}^\infty$, defined as above, which converges to a distribution function of the orthogonality measure $\mu$ as presented in $\rm (i)$.

    \item Given that $\nu>0$, the interval $\qb{-1/j_{\nu,1}', 1/j_{\nu,1}'}$ is the convex hull of the support of $\mu$ where $j'_{\nu,1}$ denotes the smallest positive zero of $J'_\nu$, and thus supp$(\mu)\subseteq \qb{-1/j_{\nu,1}', 1/j_{\nu,1}'}$.
\end{enumerate}
\end{remark}

Recall that for $n\ge1$,
\begin{equation}\label{q Recur}
    xq_{n,\nu}(x) = q_{n+1,\nu}(x) + \beta_n q_{n-1,\nu}(x),
\end{equation}
where $\beta_n = 1/[4(\nu+n-1)(\nu+n)]$ and
\begin{equation}\label{q Recur2}
    q_{0,\nu}(x) =1,\quad q_{1,\nu}(x) = \frac{x}{2}.
\end{equation}
We introduce a sequence of polynomials $\cb{q_{n,\nu}^*(x)}_{n=0}^\infty$, defined as
\begin{equation}\label{2nd}
    q_{n,\nu}^*(x) = \int_\R \frac{q_{n,\nu}(x)- q_{n,\nu}(y)}{x-y} d\mu(y),
\end{equation}
in which $\mu$ represents the orthogonality measure as in Remark \ref{remark:2.1} (i).
It is straightforward to see that $\cb{q_{n,\nu}^*(x)}_{n=0}^\infty$ satisfies the recurrence \eqref{q Recur} with initial $q_{0,\nu}^*(x)= 0$ and $q_{1,\nu}^*(x)= 1$. Hence it can be rewritten as
\begin{equation}\label{q*}
    q_{n,\nu}^*(x) = 2\nu N_n(x) = \frac{2\nu}{2^n (\nu)_{n}} R_{n-1,\nu+1}(1/x),
\end{equation}
for $n\ge0$ and $\nu>0$.

Combining all the results described above, we establish Markov-type theorem analogous to \cite[Theorem 2.6.2]{Ismail1} as follows:
\begin{theorem}\label{thm:2.3}
Let $\nu>0$. Then for $x\not\in \qb{-1/j_{\nu,1}', 1/j_{\nu,1}'}$,
\begin{equation*}
    \frac{2\nu J_\nu(1/x)}{J_\nu'(1/x)} = \int_{-1/j_{\nu,1}'}^{1/j_{\nu,1}'} \frac{d\mu(t)}{x-t}. 
\end{equation*}
\end{theorem}

\begin{proof}
Let $x_{n,1}<\cdots<x_{n,n}$ be the zeros of $q_{n,\nu}$. Let us consider the partial fraction expansion
\begin{align*}
    \frac{q_{n,\nu}^*(x)}{q_{n,\nu}(x)} = \sum_{k=1}^n \frac{c_k}{x-x_{n,k}}.
\end{align*}
To determine $c_k$'s, we consider the closed contour $\gamma_k$ whose interior contains only $x_{n,k}$ for each $1\le k\le n$, and thus by taking contour integral both sides, it follows
\begin{equation*}
    c_k = \frac{q_{n,\nu}^*(x_{n,k})}{q_{n,\nu}'(x_{n,k})}.
\end{equation*}
We now examine the value of $q_{n,\nu}^*(x_{n,k})$. It is immediate from \eqref{2nd} that
\begin{equation*}
    q_{n,\nu}^*(x_{n,k}) = - \int_\R \frac{q_{n,\nu}(y)}{x_{n,k}-y}d\mu(y).
\end{equation*}
On the other hand, by substituting $n\mapsto n-1$ and $x=x_{n,k}$ in \eqref{CD1}, we find that
\begin{align*}
    \sum_{m=0}^{n-1}  \epsilon_m \frac{q_{m,\nu}(x_{n,k})q_{m,\nu}(y)}{\lambda_m}
    = -\frac{q_{n,\nu}(y)q_{n-1,\nu}(x_{n,k})}{\lambda_{n-1}(x_{n,k}-y)}. 
\end{align*}
Integrating both sides with respect to the orthogonality measure $\mu$, it follows
\begin{equation*}
    \int_\R \frac{q_{n,\nu}(y)}{x_{n,k}-y}d\mu(y) = -\frac{\lambda_{n-1}}{q_{n-1,\nu}(x_{n,k})}.
\end{equation*}
Hence we obtain $c_k = {\lambda_{n-1}}/\qb{q_{n,\nu}'(x_{n,k})q_{n-1,\nu}(x_{n,k})} = \rho(x_{n,k})$, leading to
\begin{equation*}
    \frac{q_{n,\nu}^*(x)}{q_{n,\nu}(x)} = \sum_{k=1}^n \frac{\rho(x_{n,k})}{x-x_{n,k}} = \int_\R \frac{d\psi_n(t)}{x-t}.
\end{equation*}
Regarding Remark \ref{remark:2.1} (ii) and (iii) (see also \cite[Theorem 1.2.1]{Ismail1}), there exists a subsequence of functions $\{\psi_{n_j}\}_{j=1}^\infty$ of $\cb{\psi_n}_{n=1}^\infty$ such that
\begin{equation*}
    \lim_{j\to \infty} \int_\R \frac{d\psi_{n_j}(t)}{x-t} = \int_{-1/j_{\nu,1}'}^{1/j_{\nu,1}'} \frac{d\mu(t)}{x-t}.
\end{equation*}
Therefore, by using \eqref{limit1}, \eqref{limit2} and \eqref{q*}, we conclude the completeness of the proof.
\end{proof}

Theorem \ref{thm:2.3} demonstrates that the Stieltjes transform of $\mu$ serves as a meromorphic function $2\nu J_{\nu}(1/x)/J_{\nu}'(1/x)$ on $\R\setminus\cb{0}$. Therefore, the Perron-Stieltjes inversion formula enables us to derive the explicit form of the orthogonality measure.

\begin{theorem}\label{thm:2.4}
    Let $\nu>0$. The orthogonality measure $\mu$ with respect to $\cb{q_{n,\nu}}_{n=0}^\infty$ is purely discrete. Moreover, $\mu$ is supported only on $\big\{\pm 1/j_{\nu,k}'\big\}_{k=1}^\infty$ with
    \begin{equation}\label{mass}
        \mu\rb{ \{1/j_{\nu,k}'\}} = \mu\rb{ \{ - 1/j_{\nu,k}'\}} = \frac{2\nu}{\big(j_{\nu,k}'\big)^2-\nu^2}.
    \end{equation}
\end{theorem}

\begin{proof}
Let $\nu>0$ be fixed. From Theorem \ref{thm:2.3}, we define
\begin{equation*}
    S(x) := \frac{2\nu J_\nu(1/x)}{J_\nu'(1/x)} = \int_{-1/j_{\nu,1}'}^{1/j_{\nu,1}'} \frac{d\mu(t)}{x-t}.
\end{equation*}
Since, with the exception of the origin, $J_\nu$ and $J_\nu'$ do not have zeros in common (see \cite[\S 15.21]{Watson}) and $J_\nu'$ has only real zeros $\big\{ \pm j_{\nu,k}' \big\}_{k=1}^\infty$ for $\nu>0$, the function $S(x)$ is meromorphic on $\R\setminus\cb{0}$ with simple poles at $x= \pm 1/j_{\nu,k}'$, $k\in\mathbb{N}$, along with an essential singularity at $x=0$. Moreover, the function $S(x)$ stands for Stieltjes transform of orthogonality measure $\mu$. According to Perron-Stieltjes inversion formula (see \cite[Remark 1.2.1]{Ismail1})
\begin{equation*}
    \mu(t) -\mu(s) = \lim_{\epsilon \to 0+} \int_s^t \frac{S(x-i\epsilon)-S(x+i\epsilon)}{2\pi i}dx,
\end{equation*}
it is apparent that $\mu$ is purely discrete measure supported on a compact set. Furthermore, it turns out that $\mu$ has an isolated atom at a pole of $S(x)$, and the residue at this pole corresponds to the mass of the atom. In particular, when considering the signs $\pm$ in the same order, we have
\begin{equation*}
    \mu\rb{ \{\pm 1/j_{\nu,k}'\}} = \text{Res}\qb{ \frac{2\nu J_\nu(1/x)}{J_\nu'(1/x)}\,;\, x=\pm 1/j_{\nu,k}' } = -\frac{2\nu J_\nu(\pm j_{\nu,k}')}{\big(j_{\nu,k}'\big)^2 J_\nu''(\pm j_{\nu,k}')}.
\end{equation*}
This implies \eqref{mass}, as the differential equation for the Bessel functions $J_\nu$ yields $(j_{\nu,k}')^2 J_\nu''(\pm j_{\nu,k}') = -\big((j_{\nu,k}')^2-\nu^2\big) J_\nu(\pm j_{\nu,k}')$.

To determine the mass of $\mu$ at $x=0$, we refer readers to \cite[Theorem 2.5.6]{Ismail1} and references therein. This theorem states $\mu$ has an atom at $x=0$ if and only if the series $\sum_{n=1}^\infty {q_{n,\nu}^2(0)}/{\lambda_n}$ converges. It follows from \eqref{q Recur} that $q_{2k-1,\nu}(0)=0$ and $q_{2k,\nu}(0)= (-1)^k \beta_1\beta_3\cdots \beta_{2k-1}$. Consequently, we find
\begin{equation*}
    \sum_{k=1}^\infty \frac{q_{k,\nu}^2(0)}{\lambda_k} = \sum_{k=1}^\infty \frac{\beta_1\beta_3\cdots \beta_{2k-1}}{\beta_2\beta_4\cdots \beta_{2k}} = \sum_{k=1}^\infty \frac{\nu+2k}{\nu},
\end{equation*}
which diverges. Hence, we have $\mu(\cb{0})=0$, and the proof is complete.
\end{proof}

According to Theorem \ref{thm:2.4}, the identity \eqref{ortho} can be reformulated as:

\begin{corollary}\label{coro:2.2}
Let $\nu>0$. Then we have
\begin{equation*}
    \sum_{k\in \mathbb{Z}\setminus\cb{0}} \frac{2\nu}{\big(j_{\nu,k}'\big)^2-\nu^2}\,q_{j,\nu}(1/j_{\nu,k}')q_{s,\nu}(1/j_{\nu,k}') = \frac{\lambda_j}{\epsilon_j} \delta_{j,s},
\end{equation*}
for nonnegative integers $j,s$.
\end{corollary}

\section{Orthogonal polynomials \texorpdfstring{$p_{n,\nu}$}{p n nu} with Rayleigh-type sum as its moments}

Regarding Theorem \ref{thm:2.4}, for fixed $\nu>0$, we introduce a moment functional $\mathcal{I}$ associated with the orthogonality measure $\mu$ as follows
\begin{equation*}
    \mathcal{I}(f) = \int_\R f(x) d\mu(x) = \sum_{k\in \mathbb{Z}\setminus\cb{0}} \frac{2\nu}{\big(j_{\nu,k}'\big)^2-\nu^2}\,f(1/j_{\nu,k}').
\end{equation*}
Given the fact (see \cite[\S 15.3]{Watson})
\begin{equation}\label{ineq}
    \nu<j_{\nu,1}'<j_{\nu,2}'<\cdots \for \nu>0,
\end{equation}
it is evident that $\mathcal{I}$ is positive definite for $\nu>0$. Consequently, $\mathcal{I}$ produces the unique (up to nonzero factors) sequence of orthogonal polynomials (see \cite[\S 2, 3]{Chihara}), precisely $\cb{q_{n,\nu}}_{n=1}^\infty$, as detailed in Corollary \ref{coro:2.2}.

Of particular interest is the moment functional, for fixed $\nu>0$, defined as
\begin{equation*}
    \mathcal{L}(f) = \frac{1}{2\nu} \int_\R f(x)(1-\nu^2 x^2) d\mu(x) = \sum_{k\in \mathbb{Z}\setminus\cb{0}} \frac{1}{\big(j_{\nu,k}'\big)^2}\,f(1/j_{\nu,k}'),
\end{equation*}
whose moment $\mu_n$ of order $n$ corresponds to the Rayleigh-type sum with respect to the zeros of $J_\nu'$. To be specific, we have $\mu_n := \mathcal{L}(x^n) = \sigma_\nu'(n+2)$, 
for each $n\in \mathbb{N}\cup \cb{0}$. It is worth noting that, owing to \cite[Corollary 3]{Kytmanov}, the Rayleigh-type sum $\sigma_\nu'(n)$ can be computed as
\begin{equation}\label{RL2}
    \sigma_\nu'(n) = (-1)^n \dm{ c_1 & c_0 & 0 & \cdots & 0\\
    2 c_2 & c_1 & c_0 & \cdots & 0 \\
    3 c_3 & c_2 & c_1 & \cdots & 0 \\
    \vdots & \vdots & \vdots & \ddots & \vdots\\
    n c_{n} & c_{n-1} & c_{n-2} & \cdots & c_1},\quad n=2,3,\cdots,
\end{equation}
where $c_n$ denotes the coefficient of $x^n$ in the \emph{normalized} Bessel derivative, given by
\begin{equation}\label{series}
    2^\nu \Gamma(\nu)x^{1-\nu}J_\nu'(x) = \sum_{k=0}^\infty \frac{(-1)^k(\nu/2+1)_k}{k! 4^k (\nu/2)_k (\nu+1)_k} x^{2k}.
\end{equation}
More precisely, $c_{2k+1} = 0$ and $c_{2k} = (-1)^k(\nu/2+1)_k/[k! 4^k (\nu/2)_k (\nu+1)_k ] $ for each $k\in\mathbb{N}\cup\cb{0}$. 
\begin{remark}\label{rem:3.1}
Following the notation in \cite{Sneddon}, we put
\begin{equation*}
    S_{2n,\nu}' = \sum_{k=1}^\infty \frac{1}{\big(j_{\nu,k}'\big)^{2n}\big[\big(j_{\nu,k}'\big)^2-\nu^2\big]} = \frac{1}{4\nu}\mathcal{I}\rb{x^{2n}}
\end{equation*}
for $n=0,1,\cdots$.
Applying Corollary \ref{coro:2.2} with $j=s=0$, it follows
\begin{equation}\label{S0}
    2S_{0,\nu}' = \sum_{k\in \mathbb{Z}\setminus\cb{0}} \frac{1}{\big(j_{\nu,k}'\big)^2-\nu^2} = \frac{1}{\nu} \for \nu >0.
\end{equation}
Meanwhile, one may compute $\mathcal{L}(x^{2n-2}) = \sigma_\nu'(2n)$ via \eqref{RL2} and \eqref{series}. For instance, we obtain that
\begin{equation}\label{sigma}
    \begin{split}
        \sigma_\nu'(2) &= \frac{\nu+2}{2\nu(\nu+1)},\quad \sigma_\nu'(4) = \frac{\nu^2+8\nu + 8}{8\nu^2(\nu+1)^2(\nu+2)},\\
        \sigma'_\nu(6)&=\frac{\nu^3+16\nu^2+38\nu+24}{16\nu^3(\nu+1)^3(\nu+2)(\nu+3)}.
    \end{split}
\end{equation}
Moreover, an elementary identity $\sigma_\nu'(2n) = 2S_{2n-2,\nu}'-2\nu^2 S_{2n,\nu}'$, along with \eqref{RL2}, \eqref{S0}, gives $S_{2m,\nu}'$ for $m\ge1$ in an iterative manner. Note that there is a typographical error in \cite{Sneddon}; the equation (76) should read
\begin{equation*}
    S'_{6,\nu}=\frac{5\nu^2+16\nu+12}{32\nu^4(\nu+1)^3(\nu+2)(\nu+3)} = \frac{5\nu+6}{32\nu^4(\nu+1)^3(\nu+3)}.
\end{equation*}

\end{remark}

As $j_{\nu,k}'>\nu$ for $\nu>0$ and $k\ge1$, it is evident that the weight $1-\nu^2x^2$ remains positive over $\text{supp}(\mu) = \{ \pm 1/j_{\nu,k}' \}_{k=1}^\infty$, leading that $\mathcal{L}$ is positive definite. We observe that $\mathcal{L}$ being positive definite ensures the existence of a monic sequence of orthogonal polynomials with respect to $\mathcal{L}$, denoted as $\cb{p_{n,\nu}}_{n=0}^\infty$. Additionally, $\mathcal{L}$ is symmetric, implying $\mu_n = 0$ for odd $n$. The objective of this section is to elucidate explicit forms of the sequence $\cb{p_{n,\nu}}_{n=0}^\infty$ and its recurrence relation.

\begin{theorem} \label{thm:3.1}
Let $\nu>0$, and let $\cb{p_{n,\nu}}_{n=0}^\infty$ be a monic sequence of orthogonal polynomials with respect to $\mathcal{L}$. We have
\begin{equation}\label{p poly}
    p_{n,\nu}(x) = \frac{2}{q_{n,\nu}(1/\nu)}\frac{ q_{n,\nu}(1/\nu)q_{n+2,\nu}(x)- q_{n,\nu}(x)q_{n+2,\nu}(1/\nu)}{x^2-\nu^{-2}} ,
\end{equation}
for $n\in \mathbb{N}\cup \cb{0}$.

\end{theorem}

\begin{proof}
It is obvious from \eqref{ineq} and Theorem \ref{thm:2.4} that the weight function $x^2-\nu^{-2}$ is strictly negative on the support of $d\mu(x)$ when $\nu>0$.
Applying the same reasoning in the proof of \cite[Theorem 2.7.1]{Ismail1} with $\int_\R (x^2-\nu^{-2})d\mu(x)<0$, we deduce that
\begin{equation*}
        p_{n,\nu}(x) = \frac{C_n}{x^2-\nu^{-2}} \dm{ q_{n,\nu}(1/\nu) & q_{n+1,\nu}(1/\nu) & q_{n+2,\nu}(1/\nu) \\
        q_{n,\nu}(-1/\nu)  & q_{n+1,\nu}(-1/\nu) & q_{n+2,\nu}(-1/\nu)\\
        q_{n,\nu}(x) & q_{n+1,\nu}(x) & q_{n+2,\nu}(x)},
\end{equation*}
where $\cb{q_{n,\nu}}_{n=0}^\infty$ denotes a sequence of orthogonal polynomials with respect to $\mu(x)$ and 
\begin{equation*}
    C_n = \frac{1}{2}\dm{ q_{n,\nu}(1/\nu) & q_{n+1,\nu}(1/\nu) \\
q_{n,\nu}(-1/\nu) & q_{n+1,\nu}(-1/\nu) }.
\end{equation*}
In view of $q_{2k,\nu}(x) = q_{2k,\nu}(-x)$ and $q_{2k+1,\nu}(x) = - q_{2k+1,\nu}(-x)$, $k = 0,1,\cdots$, the desired result is immediate.

\end{proof}

\begin{theorem}\label{thm:3.2}
Let $\nu>0$. The sequence of orthogonal polynomials $\cb{p_{n,\nu}}_{n=0}^\infty$, as defined above, satisfies the three-term recurrence relation as follows
\begin{equation}\label{p recur}
    p_{n,\nu}(x)-xp_{n-1,\nu}(x)+\gamma_n p_{n-2,\nu}(x)=0, \quad n=1,2,\cdots,
\end{equation}
where we define $p_{-1,\nu}(x) = 0$, $\gamma_1 = (\nu+2)/[2\nu(\nu+1)]$ and
\begin{equation}\label{p coeef}
    \gamma_n = \frac{1}{4(\nu+n-1)(\nu+n-2)}\frac{q_{n+1,\nu}(1/\nu)q_{n-2,\nu}(1/\nu)}{q_{n,\nu}(1/\nu)q_{n-1,\nu}(1/\nu)},
\end{equation}
for $n=2,3,\cdots$, with initial conditions $p_{0,\nu}(x)=1$ and $p_{1,\nu}(x)=x$. 
\end{theorem}

\begin{proof}
As readily verified, $p_{0,\nu}(x)=1$ and $p_{1,\nu}(x)=x$ follow from Theorem \ref{thm:3.1}.
Since the moment functional $\mathcal{L}$ is positive definite and symmetric for $\nu>0$, by applying \cite[Theorem 4.1 and Theorem 4.3]{Chihara}, there exist $\gamma_n \ne 0$, $n=1,2,\cdots$, such that the three-term recurrence \eqref{p recur} for $\cb{p_{n,\nu}}_{n=0}^\infty$ holds true, where we define $p_{-1,\nu}(x) = 0$ and $\gamma_1 = \sigma_\nu'(2)$. From \eqref{sigma}, we note that $\sigma_\nu'(2) = (\nu+2)/[2\nu(\nu+1)]$.

To determine $\gamma_n$ for $n=2,3,\cdots$, we proceed by using \eqref{p poly} to find that
\begin{align*}
    &\frac{x^2-\nu^{-2}}{2}\big(p_{n,\nu}(x) - xp_{n-1,\nu}(x) \big)\\
    &= q_{n+2,\nu}(x) - xq_{n+1,\nu}(x) - \frac{q_{n+2,\nu}(1/\nu)}{q_{n,\nu}(1/\nu)} q_{n,\nu}(x) + \frac{q_{n+1,\nu}(1/\nu)}{q_{n-1,\nu}(1/\nu)}x q_{n-1,\nu}(x).
\end{align*}
Using \eqref{q Recur} twice, once as itself and once for $x=1/\nu$, we find that
\begin{equation*}
     q_{n+2,\nu}(x) - xq_{n+1,\nu}(x) - \frac{q_{n+2,\nu}(1/\nu)}{q_{n,\nu}(1/\nu)} q_{n,\nu}(x) = -\frac{q_{n+1,\nu}(1/\nu)}{q_{n,\nu}(1/\nu)} \frac{1}{\nu}\, q_{n,\nu}(x),
\end{equation*}
and by applying \eqref{q Recur} again, we have
\begin{multline*}
    \frac{q_{n+1,\nu}(1/\nu)}{q_{n-1,\nu}(1/\nu)}x q_{n-1,\nu}(x) \\
    = \frac{q_{n+1,\nu}(1/\nu)}{q_{n-1,\nu}(1/\nu)} \rb{q_{n,\nu}(x) + \frac{1}{4(\nu+n-1)(\nu+n-2)}q_{n-2,\nu}(x)}.
\end{multline*}
Combining the above identities and applying \eqref{q Recur} with $x=1/\nu$, we obtain
\begin{align*}
    &\frac{x^2-\nu^{-2}}{2}\big(p_{n,\nu}(x) - xp_{n-1,\nu}(x) \big)\\
    &= \frac{q_{n+1,\nu}(1/\nu)}{q_{n,\nu}(1/\nu)q_{n-1,\nu}(1/\nu)} \rb{ q_{n,\nu}(1/\nu) - \frac{1}{\nu}q_{n-1,\nu}(1/\nu)}q_{n,\nu}(x) \hspace{5em}\\
    & \pushright{ + \frac{q_{n+1,\nu}(1/\nu)}{q_{n-1,\nu}(1/\nu)} \frac{1}{4(\nu+n-1)(\nu+n-2)}q_{n-2,\nu}(x)}\\
    &= -\frac{q_{n+1,\nu}(1/\nu)q_{n-2,\nu}(1/\nu)}{q_{n,\nu}(1/\nu)q_{n-1,\nu}(1/\nu)} \frac{1}{4(\nu+n-1)(\nu+n-2)} q_{n,\nu}(x)\\
    & \pushright{ + \frac{q_{n+1,\nu}(1/\nu)}{q_{n-1,\nu}(1/\nu)} \frac{1}{4(\nu+n-1)(\nu+n-2)}q_{n-2,\nu}(x)}\\
    &=-\frac{x^2-\nu^{-2}}{2} \gamma_n\, p_{n-2,\nu}(x),
\end{align*}
in which the last identity follows from \eqref{p poly}.
\end{proof}

Some of examples for $p_{n,\nu}$ and $q_{n,\nu}$ are listed in Table \ref{table1} below:

\begin{table}[H]
    \centering
    \begin{tabular}{c|c|c}
        $n$ & $p_{n,\nu}(x)$ & $q_{n,\nu}(x)$ \\ \hline \hline
        $0$ & $1$ & $1$ \\ [0.5em]
        $1$ & $x$ & $\displaystyle \frac{x}{2}$\\ [1em]
        $2$ & $\displaystyle x^2-\frac{\nu^2+8\nu+8}{4\nu(\nu+1)(\nu+2)^2}$ & $\displaystyle \frac{x^2 }{2}-\frac{1}{4\nu {\left(\nu +1\right)}}$ \\ [1em]
        $3$ & $\displaystyle x^3-\frac{\nu^3+16\nu^2+38\nu+24}{2\nu(\nu+1)(\nu+3)(\nu^2+8\nu+8)}x$ & $\displaystyle \frac{x^3 }{2}-\frac{3\nu +4}{8\nu {\left(\nu +1\right)}{\left(\nu +2\right)}}x$
    \end{tabular}
    \caption{The first few expressions for $p_{n,\nu}$ and $q_{n,\nu}$}
    \label{table1}
\end{table}

It is important point to note here that the results \eqref{P1} and \eqref{P2} induced in classical theory align with the results in Theorem \ref{thm:3.1} and \ref{thm:3.2}, respectively. In particular, the comparison of Theorem \ref{thm:3.2} with \eqref{P2} suggests that the Hankel determinant $\Delta_n$ can be expressed in terms of $\gamma_k$. More specifically, we establish the following theorem:

\begin{theorem}\label{thm:3.3}
    Let $\nu \in \R \setminus\cb{0,-1,-2,\cdots}$. Then we have, for each $n\in \mathbb{N}$,
    \begin{equation}\label{Hankel1}
        \Delta_n =  \frac{\nu^2 q_{n+1,\nu}(1/\nu) q_{n+2,\nu}(1/\nu)}{ 2^{n^2-2}(\nu+1)^{2n-1} (\nu+2)^{2n-3}  \cdots (\nu+n-1)^3 (\nu+n)}.
    \end{equation}
\end{theorem}

\begin{proof}
    It is plain from \cite[Theorem 4.2 (a) and (b)]{Chihara} and Theorem \ref{thm:3.2} to find that
    \begin{equation*}
        \frac{\Delta_n}{\Delta_{n-1}} = \mathcal{L}\big( p_{n,\nu}^2(x)\big) = \prod_{k=1}^{n+1} \gamma_k,
    \end{equation*}
    in which the first equality comes from $\mathcal{L}\rb{ p_{0,\nu}^2(x)} = \gamma_1 = \sigma'_\nu(2)$. On account of \eqref{p coeef}, we have that for $n\ge1$,
    \begin{equation*}
        \prod_{k=1}^{n+1} \gamma_k= \sigma_\nu'(2) \prod_{k=2}^{n+1} \gamma_k = \frac{\sigma_\nu'(2)}{4^{n}(\nu+1)_{n}(\nu)_{n}}  \frac{q_{0,\nu}(1/\nu)q_{n+2,\nu}(1/\nu)}{q_{2,\nu}(1/\nu)q_{n,\nu}(1/\nu)}.
    \end{equation*}
    Hence, combining the above results with expressions in \eqref{sigma} and Table \ref{table1},
    the desired result in \eqref{Hankel1} has been established, provided that $\nu>0$.

    On considering \eqref{RL2} and \eqref{series}, it follows that $\sigma_\nu'(k)$ is holomorphic in terms of $\nu \in \R \setminus\cb{0,-1,-2,\cdots}$, implying that the Hankel determinant $\Delta_m = \det(\mu_{i+j-2})_{1\le i,j\le m+1}$, where $\mu_n = \sigma_\nu'(n+2)$ for $m\in \mathbb{N}\cup\cb{0}$, is also holomorphic function in $\nu \in \R \setminus\cb{0,-1,-2,\cdots}$. On the other hand, we observe that for each $k\ge1$, $q_{k,\nu}(1/\nu)$ is holomorphic in terms of $\nu \in \R \setminus\cb{0,-1,-2,\cdots}$ by making use of \eqref{q Recur}, \eqref{q Recur2}, and hence so is the right-hand expression in \eqref{Hankel1}. Therefore, by applying the identity theorem in complex analysis, the condition $\nu>0$ for \eqref{Hankel1} can be relaxed as $\nu \in \R \setminus\cb{0,-1,-2,\cdots}$, and the proof is now complete.

\end{proof}

\section{The number of complex zeros of \texorpdfstring{$J_\nu'$}{J nu'}}

The consistent context of Theorem \ref{thm:B} and Theorem \ref{thm:C} (iii) suggest that the number of complex zeros of $J_\nu'$ changes as $\nu$ crosses negative integers. Not only that, but this issue is closely related to the specific value of $\nu$ at which $J_\nu'$ has a double zero. It is supported by earlier results Theorem \ref{thm:C} (iii) and (iv). In details, $J_\nu'$ has $4$ complex zeros when $-2<\nu<\nu^*$, while it has none when $\nu^*<\nu<-1$, for some $\nu^*\in (-2,-1)$. In fact, as $\nu$ passes from the left to the particular value at which $J_\nu'$ has a double zero, it experiences a reduction of 4 complex zeros. In this section, we elaborate on all the scenarios that involve changes in the number of complex zeros.

We begin by characterizing all the specific values of $\nu$ at which $J_\nu'$ has a double zero. We refer to \cite{Kerimov1,Kerimov2} for historical background. It is straightforward from the Bessel differential equation that $J_\nu'$ achieves its double zero other than the origin only at $x = \pm \nu$. Moreover, Heinhold and Kuliseh \cite{Heinhold} have verified that the double zeros of $J_\nu'(x)$ coincide with the zeros of $J_\nu'(\nu)$ as a function of $\nu$, and vice versa. Additionally, they demonstrated that $J_\nu'(\nu)$ has no positive zeros and an odd number of negative zeros within each interval $(-k-1,-k)$, where $k\in \mathbb{N}$. To avoid ambiguity of variables for the functions, we shall specify its variables for the functions hereafter.

\begin{definition}\label{def:4.1}
    Let $\{\nu_k\}_{k=1}^\infty$ be the sequence of negative real zeros of $J'_\nu(\nu)$ arranged in the ascending order of magnitude, i.e., $|\nu_1|<|\nu_2|<\cdots$.
\end{definition}
Lorch and Muldoon \cite[Corollary 4.8]{Lorch2} verified that $J_\nu'(\nu)$ has one and only one zero within each interval $(-k-1,-k)$, where $k\in \mathbb{N}$, i.e., $-k-1< \nu_k <-k$. For a list of evaluations of $\nu_k$'s, we refer the reader to \cite[p. 122]{Heinhold}. We further elaborate on the range of these zeros as follows:

\begin{proposition}\label{prop:4.1}
As a function in $\nu$, $J'_\nu(\nu)$ has exactly one zero in each of the intervals $(-k-1/2,-k)$, where $k\in \mathbb{N}$. In other words, $-k-1/2 < \nu_k < -k$, for each $k\in \mathbb{N}$.

\end{proposition}

\begin{proof}
When $\nu$ is positive, the first positive zero of $J'_\nu(x)$ exceeds $\nu$ (\cite[\S 15.3]{Watson}). Thus, $J'_\nu(\nu)$ has no positive zeros, and we now assume $\nu$ is negative. Since the zeros of $J'_\nu(\nu)$ coincide with those of $J'_\nu(-\nu)$ due to the symmetry of the Bessel functions, we may consider the zeros of $J'_{-\nu}(\nu)$ for positive $\nu$.

From the definition of the Bessel function of the second kind
\begin{equation*}
    Y_\nu(x)=\frac{\cos{\nu\pi} J_\nu(x)-J_{-\nu}(x)}{\sin{\nu\pi}},
\end{equation*}
we have
\begin{equation}\label{Jnn}
    J'_{-\nu}(\nu)=\cos{\nu\pi}J'_\nu(\nu)-\sin{\nu\pi}Y'_\nu(\nu).
\end{equation}
On consideration behavior of $J_\nu'(x)$ and $Y_\nu'(x)$ for $x>0$ sufficiently small,
\begin{equation*}
    J_\nu'(x) \sim \frac{1}{2\Gamma(\nu)}\rb{\frac{x}{2}}^{\nu-1}>0,\quad Y_\nu'(x)\sim \frac{\Gamma(\nu+1)}{2\pi}\rb{\frac{x}{2}}^{-\nu-1}>0,
\end{equation*}
and the inequality $0< \nu< j'_{\nu,1} < y'_{\nu,1}$ (see \cite[Theorem 1, (b)]{Palmai}), we deduce that both $J'_\nu(\nu)$ and $Y'_\nu(\nu)$ are positive for $\nu>0$. Hence, by using \eqref{Jnn}, it follows that
\begin{equation*}
    (-1)^{k} J'_{-k}(k) >0,\quad (-1)^{k+1} J'_{-(k+1/2)}(k+1/2) >0.
\end{equation*}
As the function $J_{-\nu}(\nu)$ is continuous in $\nu$, it has at least one zero on each of intervals $(k,k+1/2)$, $k\in\mathbb{N}$. The uniqueness follows from the aforementioned theorem \cite[Corollary 4.8]{Lorch2}.

\end{proof}

\begin{remark}
    We note that the real entire function $x^{1-\nu}J_\nu'(x)$ is of order of growth $1$, and we refer to \cite[p. 215]{Baricz2} for detailed proof. It would be, however, helpful to note that $2m \log(2m)$ should be used in place of $m\log m$ in the second equation on page 215.
\end{remark}

Taking account of Theorem \ref{thm:3.3} and Theorem \ref{thm:A}, our main theorem can be established as follows:

\begin{theorem}\label{thm:main}
The following statements hold true:
    \begin{enumerate}[label=\rm(\roman*)]
        \item If $\nu\ge 0$ or $\nu=-1,-2,\cdots$, then $J_\nu'(x)$ has only real zeros.
        \item If $-1< \nu <0$, then $J_\nu'(x)$ has $2$ purely imaginary zeros.
        \item If $\nu_k \le \nu <-k$, $k\in \mathbb{N}$, then $J_\nu'(x)$ has $2k-2$ complex zeros.
        \item If $-k-1 < \nu <\nu_k$, $k\in \mathbb{N}$, then $J_\nu'(x)$ has  $2k+2$ complex zeros.
    \end{enumerate}
    Furthermore, if $-k-1 < \nu <-k$, $k\in \mathbb{N}\cup\cb{0}$, $J_\nu'(x)$ has exactly $2$ purely imaginary zeros for even $k$, and it has no purely imaginary zeros for odd $k$.
\end{theorem}

\noindent
The detailed proof will be given in the next section.

\begin{figure}[!ht]
  \hspace*{\fill}
  \begin{subfigure}{0.48\textwidth}
    \includegraphics[width=\linewidth]{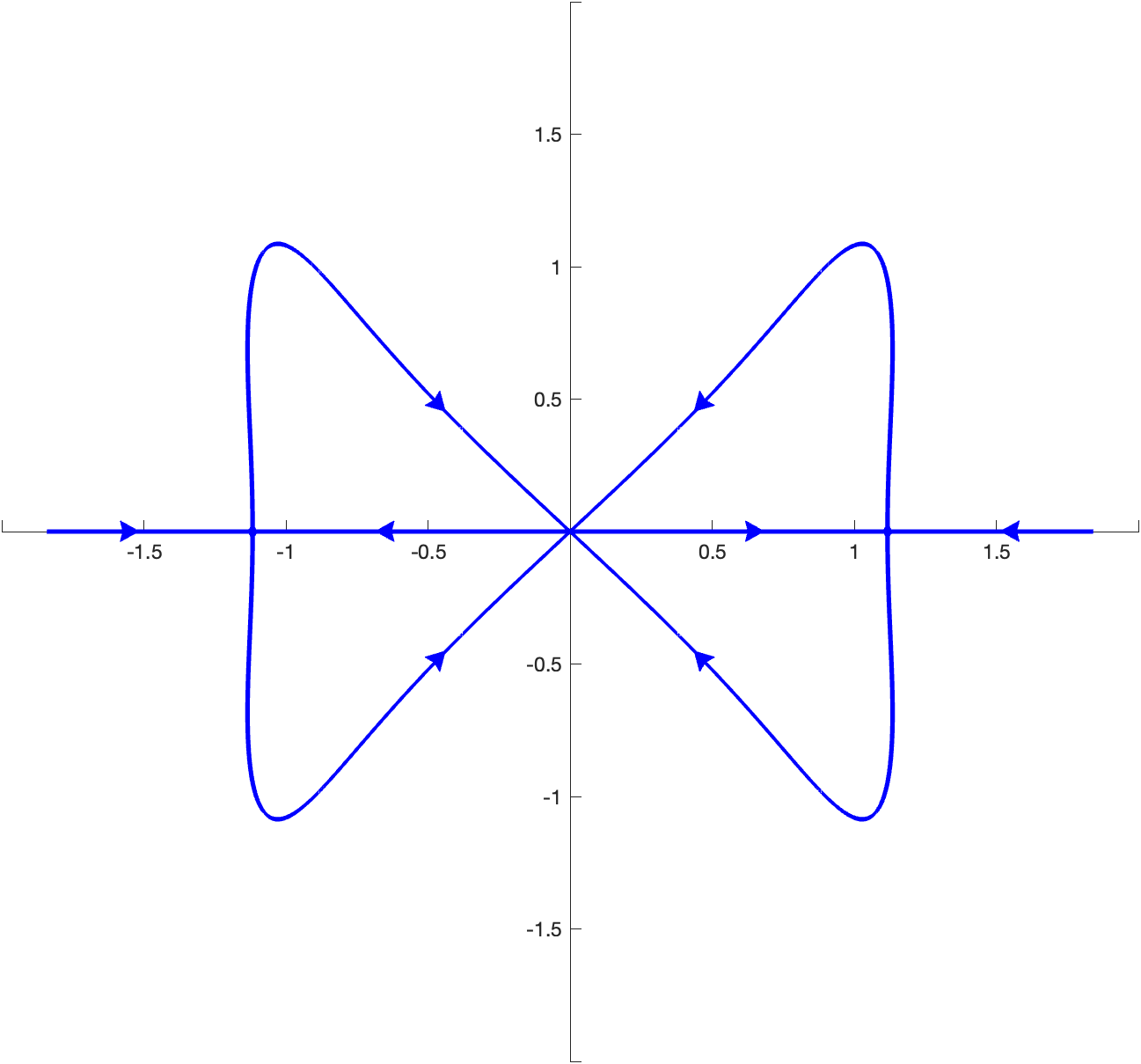}
    \caption{$\nu$ traverses from $-1$ to $-2$} \label{fig:a}
  \end{subfigure}
  \hspace*{\fill}   
  \begin{subfigure}{0.48\textwidth}
    \includegraphics[width=\linewidth]{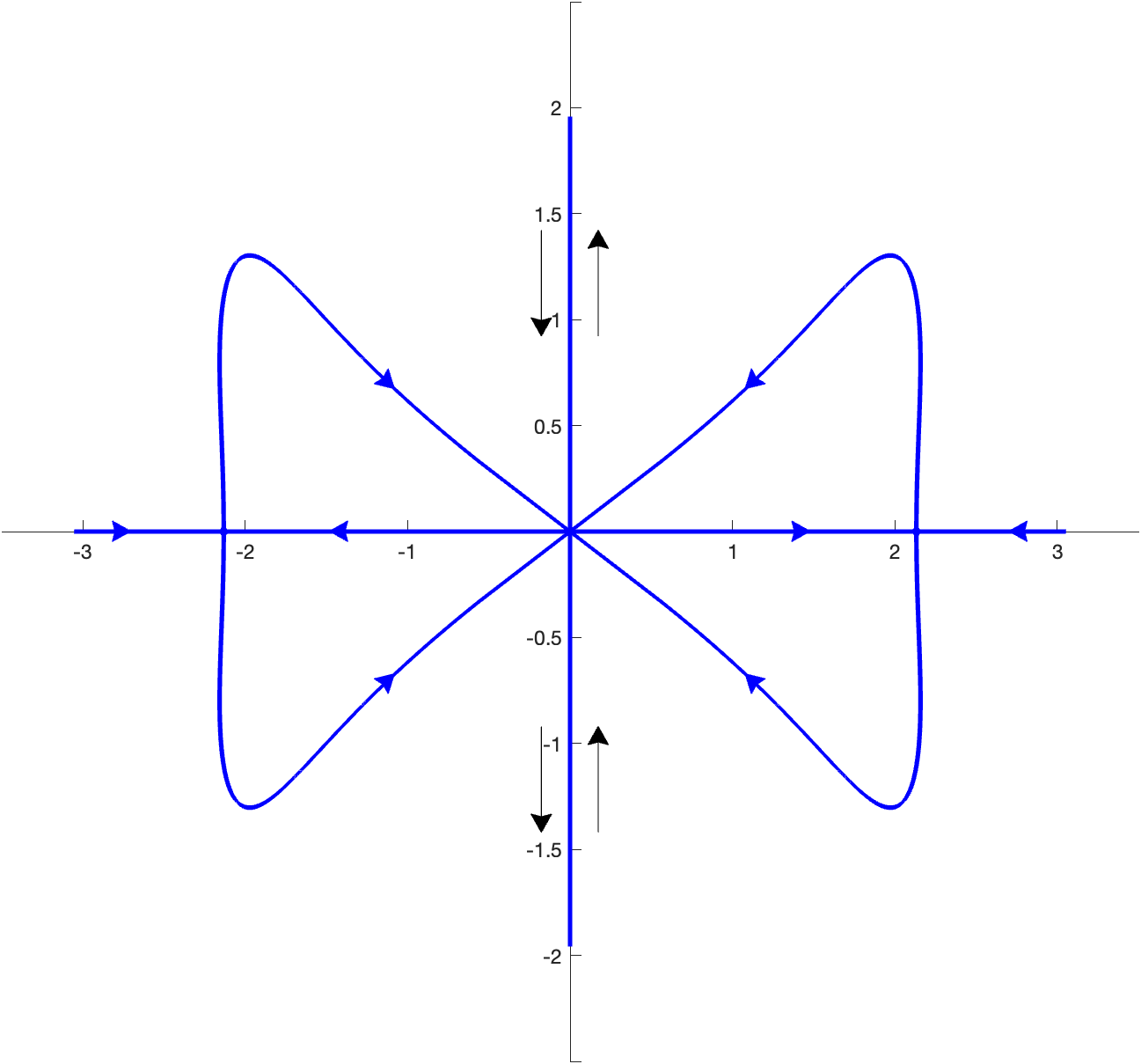}
    \caption{$\nu$ traverses from $-2$ to $-3$} \label{fig:b}
  \end{subfigure}
  \hspace*{\fill}

  \vspace{0.5em}
  \hspace*{\fill}
  \begin{subfigure}{0.48\textwidth}
    \includegraphics[width=\linewidth]{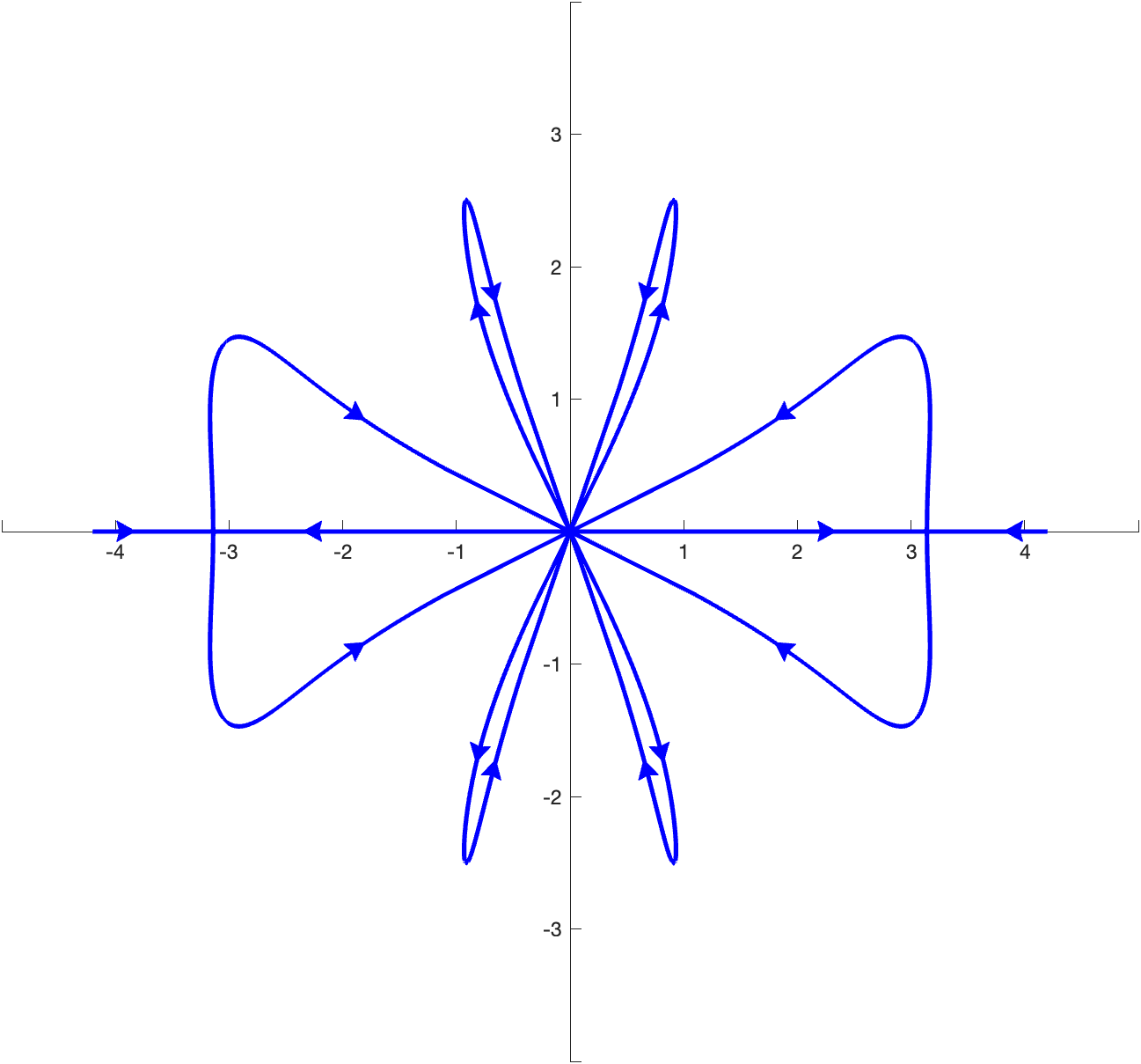}
    \caption{$\nu$ traverses from $-3$ to $-4$} \label{fig:c}
  \end{subfigure}
  \hspace*{\fill}
  \begin{subfigure}{0.48\textwidth}
    \includegraphics[width=\linewidth]{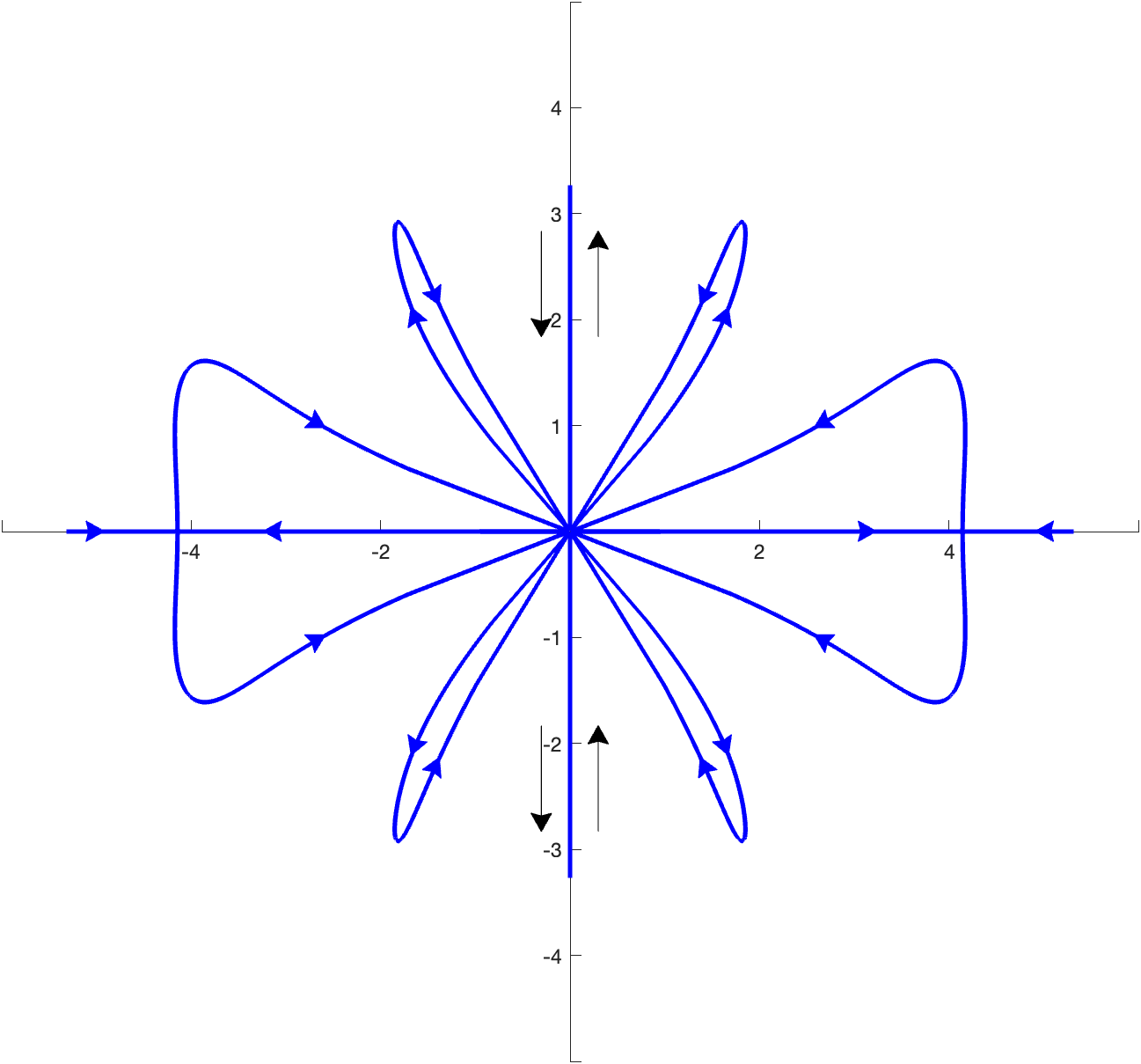}
    \caption{$\nu$ traverses from $-4$ to $-5$} \label{fig:d}
  \end{subfigure}
  \hspace*{\fill}
\caption{Trajectories of the zeros of $J_\nu'(x)$ when $\nu$ decreases continuously.} \label{fig:1}
\end{figure}

We observe from \eqref{series} that for $\nu \in \R \setminus\cb{0,-1,-2,\cdots}$,
\begin{equation}\label{limiting1}
    J_\nu'(x) \sim \frac{1}{2^\nu\Gamma(\nu)}x^{\nu-1} \quad \text{as }x \to 0.
\end{equation}
On the other hand, if $\nu=0$ or $\nu =-n$, $n\in\mathbb{N}$, then we have
\begin{equation}\label{limiting2}
    J_0'(x) \sim -\frac{1}{2}\,x,\quad J_{-n}'(x) \sim \frac{(-1)^n}{2^n \Gamma(n)} x^{n-1} \quad \text{as }x \to 0,
\end{equation}
by using well-known identities $J_\nu'(x) = \big(J_{\nu-1}(x) - J_{\nu+1}(x)\big)/2$ and $J_{-n}(x) = (-1)^n J_n(x)$, $n\in \mathbb{Z}$. Regarding \eqref{limiting1} and \eqref{limiting2}, the limiting behavior of $J_\nu'(x)$ at the origin seems to have jumps where $\nu = 0,-1,-2,\cdots$. In fact, the difference between the order of pole in \eqref{limiting1} and the multiplicity of zero in \eqref{limiting2} is exactly $2n$ when $\nu = -n$, $n\in \mathbb{N}$ (particularly, it becomes $2$ when $\nu=0$). In other words, $2n$ complex zeros stop by the origin at the moment when $\nu = -n$, $n\in \mathbb{N}$.

Therefore, considering Theorem \ref{thm:C} (iii) and Theorem \ref{thm:main}, it can be inferred that for fixed $n\in\mathbb{N}$, all complex zeros of $J_\nu'(x)$ move into the origin as $\nu$ decreases from $-n+\epsilon$ to $-n$ for small $\epsilon>0$. Indeed, as $\nu$ decreases from $-n$ to $-n-\epsilon$, it holds true:
\begin{itemize}
    \item ($n$ is even) $2$ real zeros, $2$ purely imaginary zeros and $2n-4$ complex zeros move out from the origin.
    \item ($n$ is odd) $2$ real zeros and $2n-2$ complex zeros move out from the origin.
\end{itemize}
In the special case of $n=0$, it can be also deduced that $2$ purely imaginary zeros move out from the origin. 
We refer to Figure \ref{fig:1}, which depicts the behavior of all zeros of $J_\nu'(x)$ as $\nu$ decreases continuously. 

To give further insight for the real zeros of $J_\nu'(x)$, it would be helpful to refer to \cite[\S 15.6 (4)]{Watson}, given by
\begin{equation}\label{N-W}
    \frac{d j'}{d\nu}=\frac{2 j'}{(j')^2-\nu^2}\int_0^\infty \Big((j')^2 \cosh{2t}-\nu^2\Big)K_0( 2j' \sinh{t}) e^{-2\nu t}\,dt,
\end{equation}
where $j'$ denotes any zero of $J_\nu'(x)$ and
\begin{equation}\label{MB}
    K_0(t) = \int_0^\infty e^{-t \cosh{\theta}}d\theta.
\end{equation} 
As discussed before, if $\nu>0$, $J_\nu'(x)$ has infinitely many zeros, all of which are real and symmetric with respect to the origin. Moreover,
\begin{equation*}
    0< \nu < j'_{\nu,1} < j'_{\nu,2} <\cdots.
\end{equation*}
On inspecting the formulae \eqref{N-W} and \eqref{MB}, it is immediate that the zeros $j'_{\nu,k}$ (resp. $j'_{\nu,-k}$), $k\in \mathbb{N}$, are monotonically increasing (resp. decreasing) in terms of $\nu$, provided that $\nu>0$. Consequently, as $\nu$ decreases, all zeros move towards the origin. According to  Theorem \ref{thm:C} (iii), when $\nu$ approaches $0$ from the right, two real zeros $j'_{\nu,1}$ and $j'_{\nu,-1}$ meet at the origin. Subsequently, they become purely imaginary with one zero located on the positive imaginary axis and the other on the negative imaginary axis.

As to the case where $\nu<0$, it is worth mentioning the work of Lorch and Muldoon \cite{Lorch2}. Some of the results can be summarized as follows:
\begin{itemize}
    \item If $\nu_k < \nu < -k$, $k\in \mathbb{N}$ then $j'_{\nu,n}$ is increasing in $\nu$ for each $n\ge2$ whereas $j'_{\nu,1}$ is decreasing in $\nu$, approaching to the origin.
    \item If $-k-1< \nu < \nu_k$, $k\in \mathbb{N}$, then $j'_{\nu,n}$ is increasing in $\nu$ for each $n\ge1$.
\end{itemize}
As illustrated in Figure \ref{fig:1}, when $\nu$ decreases monotonically in $(-k-1,-k)$, $k\in\mathbb{N}$, two zeros $j'_{\nu,1}$ and $j'_{\nu,2}$ come together into $x= -\nu$ at the moment of $\nu=\nu_k$ and turn to a conjugate pair of complex zeros for $-k-1<\nu <\nu_k$. Regarding the remaining case where $-1<\nu<0$, it can be inferred from the proof of \cite[Theorem 3.2]{Lorch2} that $0 < |\nu| < j'_{\nu,1} < j'_{\nu,2}<\cdots$. Moreover, it is apparent that $j'_{\nu,n}$ is increasing in $\nu$ for each $n\ge1$ and $-1<\nu<0$ by combining the above inequality with \eqref{N-W} and \cite[Corollary 4.7]{Lorch2}.

\section{Proof of Theorem \ref{thm:main}}

We begin with a proposition concerning the number of the purely imaginary zeros of $J'_\nu(x)$, depending on the range of $\nu$.

\begin{proposition}\label{prop:5.1}
    $J'_\nu(x)$ has a single pair of purely imaginary zeros for $-2k+1<\nu<-2k+2,k\in \mathbb{N}$ and no purely imaginary zeros for $-2k<\nu<-2k+1,k\in \mathbb{N}$.
\end{proposition}
\begin{proof}
    Let us consider $J'_{-\nu}(x)$ for $\nu>0$. We note that purely imaginary zeros of $J'_{-\nu}(x)$ are necessarily positive real zeros of $I'_{-\nu}(x)$ where $I_\nu$ stands for the modified Bessel function of the first kind of order $\nu$ (see \cite[\S 3.7]{Watson}). 
    From \cite[\S 3.7 (6)]{Watson}, we deduce that
    \begin{equation}\label{Inu}
        I'_{-\nu}(x)=I'_{\nu}(x)+\frac{2\sin{\nu\pi}}{\pi}K'_{\nu}(x). 
    \end{equation}
    In view of the series representation for $I_\nu(x)$ \cite[\S 3.7 (2)]{Watson} and the integral representation for $K_\nu(x)$ \cite[\S 6.22 (5)]{Watson}, we have
    \begin{equation*}
        I'_\nu(x)>0\quad\textit{and}\quad K'_\nu(x)<0,\quad x>0,\quad \nu>0,
    \end{equation*}
    and 
    \begin{equation*}
        I''_\nu(x)>0\quad\textit{and}\quad K''_\nu(x)>0,\quad x>0,\quad \nu>1.
    \end{equation*}
    The first two inequalities with \eqref{Inu} imply that $I'_{-\nu}(x)$ has no positive real zeros for $2k-1<\nu<2k, k\in \mathbb{N}$. On the other hand, the last two inequalities with \eqref{Inu} yield that $I'_{-\nu}(x)$ is an increasing function for $2k<\nu<2k+1, k\in \mathbb{N}$. On account of two limiting behaviors (\cite[\S 3.71 (2)]{Watson} and \cite[10.30.2, 10.30.4]{DLMF})
    \begin{equation*}
        K'_\nu(x)\sim -\frac{\Gamma(\nu)}{4}\left( \frac{x}{2} \right)^{-\nu-1},\quad \text{as } x\to 0+,\quad \nu>0,
    \end{equation*}
    and $I'_\nu(x) \to+\infty$ as $x\to +\infty$ for $\nu\in \R$,
    we conclude that for $2k<\nu<2k+1, k\in \mathbb{N}$, $I'_{-\nu}(x)$ has at least one zero on $(0,\infty)$ and the zero is unique since $I'_{-\nu}(x)$ is an increasing function. The remaining case where $0<\nu<1$ was covered in Theorem \ref{thm:C} (ii).
\end{proof}

Our next objective is to count the number of complex zeros of $J'_\nu(x)$ for given $\nu$. According to Theorem \ref{thm:A} in correspondence with $\sigma_n = \sigma_\nu'(n)$ and $\mathcal{D}_n = \Delta_n$, the number of complex zeros of $J_\nu'(x)$ is twice the number of negative elements of a sequence $\{ \Lambda_n\}_{n=0}^\infty$ where
\begin{equation*}
    \Lambda_n=\Delta_{n-1} \Delta_{n},\quad n=0,1,\cdots.
\end{equation*}
We abuse notations $\Delta_n$ and $\Lambda_n$ but those depend on the value of $\nu$.
To facilitate our analysis on $\Lambda_n$, we introduce auxiliary functions
\begin{equation*}
    h_n(\nu)=2^n (\nu)_n q_{n,\nu}(1/\nu)
    =\frac{R_{n,\nu}(\nu)-R_{n-2,\nu+2}(\nu)}{2},
\end{equation*}
for $n=0,1,\cdots$.
Note that the zeros of $h_n(\nu)$ coincide with those of $q_{n,\nu}(1/\nu)$ unless $\nu$ is a nonpositive integer, it also has the following properties: It is immediate from \eqref{Lommel Recur} that $\{h_n(\nu)\}_{n=1}^\infty$ satisfies the recurrence relation
\begin{equation}\label{Lommel Recur2}
    h_{n-1}(\nu)+h_{n+1}(\nu)=\frac{2(\nu+n)}{\nu}h_{n}(\nu),\quad n=1,2,\cdots, 
\end{equation}
with initial conditions $h_0(\nu)=h_1(\nu)=1$. Moreover, it can be inductively deduced from the above relation that for $n=0,1,\cdots,$
\begin{equation}\label{init_behav}
    h_{n+1}(\nu)\sim \frac{2^n n!}{\nu^{n}}\quad as\quad \nu\to 0-,
\end{equation}
and
\begin{equation}\label{lim_behav}
    \lim_{\nu\to -\infty} h_n(\nu)=1.
\end{equation}

We now proceed to rephrase \eqref{Hankel1} in terms of $h_n(\nu)$ as
\begin{equation*}
    \Delta_n=\frac{h_{n+1}(\nu)h_{n+2}(\nu)}{2^{(n+1)^2} (\nu+1)^{2n+1} (\nu+2)^{2n-1}  \cdots (\nu+n)^3 (\nu+n+1)},
\end{equation*}
for $n=0,1,2,\cdots$, leading to
\begin{equation}\label{lambda}
    \Lambda_n=\frac{ \big(h_{n+1}(\nu)\big)^2  }{2^{2n^2+2n+1} (\nu+1)^{4n} (\nu+2)^{4n-4}  \cdots  (\nu+n)^4} \frac{h_{n}(\nu) h_{n+2}(\nu)}{\nu+n+1},
\end{equation}
for $n=1,2,\cdots.$ If $\nu$ is neither a zero of $h_n(\nu)$ for any $n=1,2,\cdots$ nor a nonpositive integer, then we find that
\begin{equation}\label{sign}
    \sgn(\Lambda_n)=\sgn \big( (\nu+n+1)h_{n}(\nu) h_{n+2}(\nu) \big),\quad n=0,1,\cdots.
\end{equation}
As a result, the distribution of zeros of $h_n(\nu)$ is notable because the sign change of $h_n(\nu)$ affects the sign of $\Lambda_n$.

\begin{lemma}\label{lem:5.1}
For each $n\in \mathbb{N}$, $h_n(\nu)$ has only $(n-1)$ negative simple zeros. Moreover, the zeros of $h_{n}(\nu)$ and those of $h_{n+1}(\nu)$ are interlaced.
\end{lemma}

\begin{proof}
Taking differentiation after dividing both sides of \eqref{Lommel Recur2} by $h_n(\nu)$, the Wronskian formula is given by
\begin{align*}
    W[h_{n},h_{n+1}](\nu)&= h_n(\nu)\rb{\frac{d}{d\nu}h_{n+1}(\nu)}-\rb{\frac{d}{d\nu}h_n(\nu)} h_{n+1}(\nu) \\
    &=-\frac{2n}{\nu^2}\big(h_{n}(\nu)\big)^2+W[h_{n-1},h_{n}](\nu).
\end{align*}
Applying this iteratively, we obtain
\begin{equation}\label{Wrons1}
    \begin{split}
    W[h_{n},h_{n+1}](\nu) &=-\frac{2n}{\nu^2}\big(h_{n}(\nu)\big)^2-\frac{2(n-1)}{\nu^2}\big(h_{n-1}(\nu)\big)^2\\  &\qquad - \cdots -\frac{4}{\nu^2}\big(h_{2}(\nu)\big)^2 -\frac{1}{\nu^2}\big(h_{1}(\nu)\big)^2+W[h_{0},h_{1}](\nu)\\
    &<0,\quad \nu<0,\quad n\in \mathbb{N},
\end{split}
\end{equation}
since $h_0(\nu)=h_{1}(\nu) = 1$. For $\nu<0$, $n\in \mathbb{N}$, any zeros of $h_{n}(\nu)$, if exist, are simple because $W[h_n,h_{n+1}](\nu)$ vanishes at any multiple zeros of $h_n(\nu)$.

Now let $\mu$ and $\bar{\mu}$ be any two consecutive negative zeros of $h_{n}(\nu)$ with $\mu < \bar{\mu}$. It follows that $h'_{n}(\mu)h'_{n}(\bar{\mu}) <0$ since $\mu$ and $\bar{\mu}$ are simple zeros. From the fact that $W[h_{n},h_{n+1}](\nu) <0$ for all $\nu < 0,$ we have
\begin{equation*}
  h_{n+1}(\mu) h_{n+1}(\bar{\mu}) = \frac{W[h_{n},h_{n+1}](\mu)W[h_{n},h_{n+1}](\bar{\mu})}{h'_{n}(\mu)h'_{n}(\bar{\mu})} <0,
\end{equation*}
so an odd number of zeros of $h_{n+1}(\nu)$ are located between $\mu$ and $\bar{\mu}.$ Interchanging the functions $h_{n}(\nu)$ and $h_{n+1}(\nu)$ throughout the aforementioned argument, we conclude that $h_{n}(\nu)$ has an odd number of zeros between any two consecutive zeros of $h_{n+1}(\nu)$. Therefore, the zeros of $h_{n}(\nu)$ are interlaced with those of $h_{n+1}(\nu)$.

We observe that 
$$\frac{d}{d\nu} \frac{h_{n+1}(\nu)}{ h_{n}(\nu)} = \frac{W[h_{n},h_{n+1}](\nu)}{h_{n}^2(\nu)},$$
which directly implies that $h_{n+1}(\nu) / h_{n}(\nu)$ is a strictly decreasing function for all $\nu<0$ except at points where $h_{n}(\nu) = 0$.
Furthermore, from \eqref{init_behav} and \eqref{lim_behav}, we have
\begin{equation*}
    \lim_{\nu\to 0-}\frac{h_{n+1}(\nu)}{h_{n}(\nu)}=-\infty \quad\text{and}\quad \lim_{\nu\to -\infty}\frac{h_{n+1}(\nu)}{h_{n}(\nu)}=1.
\end{equation*}
Combining these with the above interlacing property, we obtain that, for $n>1$, $h_{n+1}(\nu)$ has one more negative zero than $h_n(\nu)$ and the largest negative zero of $h_{n+1}(\nu)$ comes closer to the origin than that of $h_{n}(\nu)$.
Since $h_2(\nu)=(\nu+2)/\nu$ has one negative zero, $h_n(\nu)$ has $(n-1)$ negative zeros for $n\in \mathbb{N}$.
To show that $h_n (\nu)$ has only negative zeros, we define the polynomial $H_{n}$ as $H_{n}(\nu) = \nu^{n-1} h_n(\nu)$ for each $n \in \mathbb{N}.$ Then, the polynomial satisfies $\nu^2 H_{n}(\nu) + H_{n+2}(\nu) = 2(\nu+n+1)H_{n+1}(\nu)$ with the initial values $H_1(\nu) = 1$ and $H_2(\nu) = \nu+2$. It follows that $H_n(\nu)$ is a monic polynomial of degree $n-1$, and since the zeros of $h_n(\nu)$ coincide with those of $H_n(\nu)$, we conclude that $h_n(\nu)$ has only $(n-1)$ negative simple zeros for $n \in \mathbb{N}$.
\end{proof}
We can now define a sequence of zeros of $h_n(\nu)$ for each $n\ge2$.
\begin{definition}\label{def:5.1}
    For $n\ge 2$, let $\{\mu_{n,k}\}_{k=1}^{n-1}$ be the sequence of negative zeros of $h_{n}(\nu)$, arranged in ascending order of magnitude.
\end{definition}

\begin{remark}\label{rem:5.1} \
    \begin{enumerate}[label=(\roman*)]
        \item Adopting this definition, Lemma \ref{lem:5.1} states
        \begin{equation}\label{inter}
            \mu_{n+1,n}<\mu_{n,n-1}<\mu_{n+1,n-1}<\cdots<\mu_{n+1,2}<\mu_{n,1}<\mu_{n+1,1},
        \end{equation}
        for $n = 2,3,\cdots$.
        \item The zeros of $h_{n}(\nu)$ approach to the zeros of $J'_{\nu}(\nu)$ as $n$ gets large, by Hurwitz's theorem in complex analysis, concerning the limit of the zeros of a sequence of holomorphic functions, {\rm i.e.}, for each $k=1,2,\cdots$, we have
        \begin{equation*}
            \lim_{n\to \infty} \mu_{n,k}=\nu_k.
        \end{equation*}

        \item Regarding (i) and (ii), the sequence $\cb{\mu_{n,k}}_{n=k+1}^\infty$, for fixed $k\in \mathbb{N}$, is strictly increasing sequence, which converges to $\nu_k$. Accordingly, for given $k\in \mathbb{N}$, it follows that
        \begin{equation*}
            \mu_{n,s} < \nu_k \quad \text{for } s\ge k,\ n\ge s+1.
        \end{equation*}
    \end{enumerate}
\end{remark}
It would be more convenient to define a counting function $\mathcal{N}(\nu)$ by the number of negative elements in $\{ \Lambda_n\}_{n=0}^\infty$ for fixed $\nu$, since the sign of $\Lambda_n$ depends on $\nu$.
By means of \eqref{lambda}, it is not difficult to see that $\mathcal{N}(\nu)$ is undefined if $\nu$ belongs to one of sets $\mathbb{Z}_{\le 0},M$, and the function $\mathcal{N}(\nu)$ is constant in terms of $\nu$ on each of closed intervals lying on $\R \setminus (M\cup\mathbb{Z}_{\le0})$, where $\mathbb{Z}_{\le 0}$ denotes the set of nonpositive integers and $M$ denotes the set of all zeros of $h_{n}(\nu)$, $n\ge2$.

On the other hand, Theorem \ref{thm:C} (iii), (iv) and what has been discussed in section 4 suggest that $\mathcal{N}(\nu)$ may have discontinuities at $\nu = \nu_k$, $k=1,2,\cdots$, presented in Definition \ref{def:4.1}.

Specifically, we denote the above two sets $M, N$ as
\begin{equation*}
    M=\{ \mu_{n,k} \, :\, n\ge 2,\ 1\le k \le n-1 \},\quad N=\{ \nu_k \, :\, k\in \mathbb{N} \}.
\end{equation*}
Owing to Proposition \ref{prop:4.1}, it is immediate that the sets $N$ and $\mathbb{Z}_{\le 0}$ are disjoint.
The set $M$, however, may have nonempty intersections with $N$ and $\mathbb{Z}_{\le 0}$, respectively. For instance, $-2 \in \mathbb{Z}_{\le 0}$ is a zero of $h_2(\nu)= (\nu+2)/\nu$.

We have divided the proof into a sequence of lemmas, concerning if the value of $\mathcal{N}(\nu)$ changes as $\nu$ crosses points lying in each of the disjoint sets $M\setminus (N \cup \mathbb{Z}_{\le 0})$, $N$, and $\mathbb{Z}_{\le0}$.

\begin{lemma}\label{lem:5.2}
For $\mu\in M\setminus (N \cup \mathbb{Z}_{\le 0})$,
\begin{equation*}
    \lim_{\nu\to \mu-}\mathcal{N}(\nu)=\lim_{\nu\to \mu+}\mathcal{N}(\nu).
\end{equation*}
Thus, the number of complex zeros of $J'_\nu(x)$ does not change as $\nu$ crosses the points in $M\setminus(N\cup \mathbb{Z}_{\le 0})$.
\end{lemma}

\begin{proof}
    Let $n,\,k$ be fixed with $n\ge2$, $1\le k\le n-1$, and let
    $\mu = \mu_{n,k}$ such that $h_{m}(\mu_{n,k}) \ne 0$ for all $m\ne n$.
    By using \eqref{sign}, we note that the signs of $\Lambda_{n-2}$ and $\Lambda_{n}$ change as $\nu$ crosses $\mu$, while $\Lambda_j$ keeps the same sign near $\nu=\mu$ for $j\ne n-2,n$. 
    
    From \eqref{init_behav}, we observe that $(-1)^{n-1} h_{n}(\nu)>0$ on $(-\epsilon, 0)$ for small $\epsilon>0$. This observation, with the fact that $\mu=\mu_{n,k}$ is a simple zero, allows us to deduce
    \begin{equation}\label{sign2}
        (-1)^{n+k} h'_{n}(\mu) >0,
    \end{equation}
    which implies that
    \begin{equation}\label{sign3}
        \lim_{\nu\to \mu+} \sgn \big(h_{n}(\nu)\big)=(-1)^{n+k},\quad  \lim_{\nu\to \mu-} \sgn \big(h_{n}(\nu)\big)=(-1)^{n+k+1}.
    \end{equation}
    Furthermore, substituting $n \mapsto n-1$, $\nu=\mu$ in \eqref{Wrons1} and using \eqref{sign2}, we find that $W[h_{n-1},h_n](\mu) = h_{n-1}(\mu)h'_n(\mu)<0$ so that $(-1)^{n+k+1}h_{n-1}(\mu) >0$. Since $h_{n-1}(\nu)$ and $h_n(\nu)$ cannot have common zeros due to Lemma \ref{lem:5.1}, we also find that the following holds near $\nu = \mu$:
    \begin{equation}\label{sign4}
        (-1)^{n+k+1}h_{n-1}(\nu) >0.
    \end{equation}
    Meanwhile, we also observe from \eqref{Lommel Recur2} that
    if $\nu$ lies in the small neighborhood of $\mu$,
    \begin{equation}\label{sign1}
        \sgn\big((\nu+n-1) h_{n-2}(\nu) \big) = - \sgn\big(  h_{n-1}(\nu) \big).
    \end{equation}
    On combining \eqref{sign}, \eqref{sign1} and \eqref{sign4}, if $\nu$ lies in the small neighborhood of $\mu$ but $\nu \ne \mu$, we have
    \begin{equation*}
        \sgn(\Lambda_{n-2})=\sgn \big( (\nu+n-1)h_{n-2}(\nu) h_{n}(\nu) \big) = \sgn\big( (-1)^{n+k}h_{n}(\nu) \big).
    \end{equation*}
    Hence by \eqref{sign3}, we conclude that the sign of $\Lambda_{n-2}$ changes from positive to negative as $\nu$ crosses $\mu$ from right to left, {\rm i.e.},
    \begin{equation}\label{change1}
    \lim_{\nu\to \mu-} \sgn\rb{\Lambda_{n-2}} = -1,\quad  \lim_{\nu\to \mu+}  \sgn\rb{\Lambda_{n-2}} = 1.
    \end{equation}

    In a similar manner, one may verify $\sgn(\Lambda_{n})=\sgn\big( (-1)^{n+k+1}h_{n}(\nu) \big)$,
    leading to
    \begin{equation}\label{change2}
    \lim_{\nu\to \mu-} \sgn\rb{\Lambda_{n}} = 1,\quad  \lim_{\nu\to \mu+}  \sgn\rb{\Lambda_{n}} = -1.
    \end{equation}
    Therefore, the total number of negative elements in $\{\Lambda_{j}\}_{j=0}^\infty$ is preserved near $\nu = \mu$.

    More generally, there can be several distinct pairs of $(n,k)$ such that $\mu=\mu_{n,k}$. Let $I$ be an index set such that $\mu = \mu_{n,k}$ for all $(n,k)\in I$. In fact, $I$ is a finite set. To explain, we take $(n_0,k_0)\in I$. Then, there exists $k^* \ge k_0$ such that $\nu_{k^*+1} < \mu_{n_0,k_0} < \nu_{k^*}$. According to Remark \ref{rem:5.1} (iii), it is apparent that if $(n,k)\in I$, then $k\le k^*$ and $(n',k)\not\in I$ for any $n'\ne n$. Hence, $I$ contains at most $k^*$ elements. Let $|I| = K\le k^*$. We also claim that $|n - n'| \ne 2$ for any $(n,k),\,(n',k')\in I$. Suppose, on the contrary, that $n' = n+2$ for some $(n,k),\,(n',k')\in I$. By substituting $\mu_{n,k} = \mu_{n',k'}$ in \eqref{Lommel Recur2} with $n\mapsto n+1$, we obtain
    \begin{equation*}
        \frac{2(\mu_{n,k}+n+1)}{\mu_{n,k}} h_{n+1}(\mu_{n,k}) = h_n(\mu_{n,k}) + h_{n+2}(\mu_{n,k}) = 0.
    \end{equation*}
    Since $\mu_{n,k} \not\in \mathbb{Z}_{\le 0}$, it follows $h_{n+1}(\mu_{n,k})=0$ which contradicts to the interlacing property in Lemma \ref{lem:5.1}. 
    In this case, the sign of $\Lambda_{j}$ for $j\in J$ change as $\nu$ crosses $\mu$, while $\Lambda_n$ keeps the same sign near $\nu=\mu$ for $n\not \in J$, where $J = \cb{n_j-2,\,n_j\,:\, (n_j,k_j)\in I,\ j=1,\cdots,K}$. Obviously, $J$ has exactly $2K$ elements since $|n_i-n_j|\ne 2$. Thus, the same argument, as discussed above, can be applied to see that for each $j = 1,\cdots,K$,
    \begin{equation}\label{sign5}
    \begin{alignedat}{2}
    \lim_{\nu\to \mu-} &\sgn\rb{\Lambda_{n_j-2}} = -1,\quad  &&\lim_{\nu\to \mu+}  \sgn\rb{\Lambda_{n_j-2}} = 1,\\
    \lim_{\nu\to \mu-} &\sgn\rb{\Lambda_{n_j}} = 1,\quad  &&\lim_{\nu\to \mu+}  \sgn\rb{\Lambda_{n_j}} = -1,
    \end{alignedat}
    \end{equation}
    and the proof is now complete because the value of $\mathcal{N}(\nu)$ is invariant in the small neighborhood of $\mu$ but $\nu \ne \mu$, by the definition of $\mathcal{N}(\nu)$.
    
\end{proof}

\begin{lemma}\label{lem:5.3} 
    For $n\in \mathbb{N}$,
    \begin{equation*}
        \lim_{\nu\to -n-}\mathcal{N}(\nu)=\lim_{\nu\to -n+}\mathcal{N}(\nu)-1.
    \end{equation*}
    Thus, the number of complex zeros of $J'_\nu(x)$ decreases by $2$ as $\nu$ crosses negative integers from right to left.
\end{lemma}

\begin{proof}
    Suppose $-n \notin M$. By \eqref{sign}, obviously the sign of $\Lambda_{n-1}$ changes near $\nu=-n$ but the signs of $\Lambda_j$, $j\ne n-1$, do not change near $\nu=-n$. To explain further, we observe from the recurrence relation \eqref{Lommel Recur2} that
    \begin{equation*}
        h_{n-1}(-n)+h_{n+1}(-n)=0.
    \end{equation*}
    Thus, $h_{n-1}(\nu)$ and $h_{n+1}(\nu)$ have opposite signs on $(-n-\epsilon,-n+\epsilon)$ for small $\epsilon>0$ since $h_{n-1}(-n)\ne0$ and $h_{n+1}(-n)\ne0$. Then, by \eqref{sign}, the sign of $\Lambda_{n-1}$ changes from negative to positive as $\nu$ crosses $-n$ from right to left.
    
    Next, let us consider the case when $-n \in M$. Let $\{(m_j,k_j)\}_{j=1}^K$ be an ordered set of all pairs of positive integers satisfying $-n=\mu_{m_j,k_j}$, arranged in increasing order in $m_j$, that is, $m_1<m_2<\cdots<m_K$. We note that $K$ must be finite by the same argument in the proof of Lemma \ref{lem:5.2}. Since the zeros of $h_m(\nu)$ and $h_{m+1}(\nu)$ cannot coincide for any $m\in \mathbb{N}$, it follows $m_{j+1}-m_j>1$ for any $j=1,2,\cdots,K-1$. Moreover, by \eqref{Lommel Recur2}, $h_m(\nu)$ and $h_{m+2}(\nu)$ may have a common zero only when $\nu=-m-1$. Thus, we find that $m_{j+1}-m_j>2$ for all $j=1,2,\cdots,K-1$ with a single exception where $m_{j+1}=n+1$ and $m_j=n-1$. In view of \eqref{sign}, we observe that as $\nu$ crosses $-n$ from right to left, only elements (possibly redundant)
    \begin{equation}\label{list}
        \Lambda_{n-1}, \Lambda_{m_1-2}, \Lambda_{m_1}, \Lambda_{m_2-2}, \Lambda_{m_2}, \cdots, \Lambda_{m_K-2}, \Lambda_{m_K},
    \end{equation}
    experience sign change among the elements of $\cb{\Lambda_j}_{j=0}^\infty$.

    Suppose $m_j\neq n-1$ for all $j=1,\cdots,K$, which means that $m_{j+1}-m_j>2$ for all $j=1,\cdots,K-1$. In this case, all the elements of \eqref{list} are distinct from each other. In particular, by \eqref{change1} and \eqref{change2}, we have that
    \begin{alignat}{2}
    \lim_{\nu\to -n-} &\sgn\rb{\Lambda_{m_j-2}} = -1,\quad  &&\lim_{\nu\to -n+}  \sgn\rb{\Lambda_{m_j-2}} = 1, \label{change3}\\
    \lim_{\nu\to -n-} &\sgn\rb{\Lambda_{m_j}} = 1,\quad  &&\lim_{\nu\to -n+}  \sgn\rb{\Lambda_{m_j}} = -1, \label{change4}
    \end{alignat}
    for each $j=1,\cdots,K$. Since we know that $\Lambda_{n-1}$ changes from negative to positive, the total number of negative numbers in $\{ \Lambda_j\}_{j=0}^\infty$ decreases by $1$ as $\nu$ crosses $-n$ from right to left, implying the desired result.

    The remaining case is when $m_{\ell}=n-1$ for some $1\le \ell< K$. In the present case, based on \eqref{Lommel Recur2}, it necessarily follows $m_{\ell+1}=n+1$. Then, it is evident that $\Lambda_{n-1}=\Lambda_{m_{\ell}}=\Lambda_{m_{\ell+1}-2}$ and all the others in \eqref{list} are distinct from each other. Likewise, we find that \eqref{change3} holds true for each $j\in \cb{1,\cdots,K}\setminus\cb{\ell+1}$, as well as \eqref{change4} holds true for each $j\in \cb{1,\cdots,K}\setminus\cb{\ell}$.
    To elaborate the sign change of $\Lambda_{n-1}$ near $\nu=-n= \mu_{n-1,k_\ell} = \mu_{n+1,k_{\ell+1}}$, we observe from \eqref{inter} that 
    $\mu_{n,k_\ell+1} < \mu_{{n+1},k_{\ell+1}}= \mu_{n-1,k_\ell} < \mu_{n,k_\ell} < \mu_{n+1,k_\ell}$,
    leading to $k_{\ell+1} = k_\ell +1$. Regarding \eqref{sign} and \eqref{sign3}, we conclude that $\Lambda_{n-1}$ turns into positive from negative as $\nu$ crosses $-n$ from the right.
    Therefore, the proof is now complete.
\end{proof}

\begin{lemma}\label{lem:5.4}
For $k\in\mathbb{N}$,
\begin{equation*}
    \lim_{\nu\to \nu_{k}-}\mathcal{N}(\nu)= \lim_{\nu\to \nu_{k}+}\mathcal{N}(\nu)+2.
\end{equation*}
Thus, the number of complex zeros of $J'_\nu(x)$ increases by $4$ as $\nu$ crosses $\nu_k$ from right to left.
\end{lemma}

\begin{proof}
    Let $\nu_k \in N$ be fixed. We claim that there exists $\epsilon \in (0,1/2]$ such that 
    \begin{equation}\label{claim}
        \big[(\nu_k - \epsilon, \nu_k)\cup(\nu_k,\nu_k + \epsilon)\big] \cap \cb{ \mu_{n,\ell}\,:\, \ell\ne k,\ n\ge \ell+1 } = \emptyset.
    \end{equation}
    For the sake of contradiction, we suppose that, for any $\epsilon>0$, 
    \begin{equation*}
        B_\epsilon \cap \cb{ \mu_{n,\ell}\,:\, \ell\ne k,\ n\ge \ell+1 } \neq \emptyset,
    \end{equation*}
    where $B_\epsilon = (\nu_k - \epsilon, \nu_k)\cup(\nu_k,\nu_k + \epsilon)$.
    If we take an arbitrary $0<\epsilon\le 1/2$, by using Proposition \ref{prop:4.1} and Remark \ref{rem:5.1} (i)-(iii), it follows that
    \begin{equation*}
        \mu_{n,\ell} <\nu_{k+1} < -k-1< \nu_k-\epsilon < \mu_{m,k} < \nu_k
    \end{equation*}
    for all $\ell > k$, $n\ge \ell+1$ and $\mu_{m,k}\in B_\epsilon\cap \cb{\mu_{n,k}}_{n=k+1}^\infty$. Hence, we deduce that $B_\epsilon \cap \cb{ \mu_{n,\ell}\,:\, \ell> k,\ n\ge \ell+1  } = \emptyset$, which implies that
    \begin{equation*}
        B_\epsilon \cap \cb{ \mu_{n,\ell}\,:\, \ell< k,\ n\ge \ell+1 } \neq \emptyset,
    \end{equation*}
    for any $0<\epsilon\le 1/2$. By the pigeonhole principle and noting that
    \begin{equation*}
        \cb{ \mu_{n,\ell}\,:\, \ell< k,\ n\ge \ell+1 } = \bigcup_{1\le \ell <k} \cb{\mu_{n,\ell}}_{n=\ell+1}^\infty,
    \end{equation*}
    one may choose a subsequence $\big\{\mu_{n_{j},L}\big\}_{j=1}^\infty$ of $\big\{\mu_{n,L}\big\}_{n=L+1}^\infty$ for some $1\le L < k$ such that $\mu_{n_j,L}\to \nu_k$ as $j \to \infty$ but $\mu_{n_j,L} \ne \nu_k$ for $j\ge1$. Then, it follows from Remark \ref{rem:5.1} (ii) that
    \begin{equation*}
        \nu_k = \lim_{j\to\infty} \mu_{n_j,L} = \nu_L,\quad k\ne L,
    \end{equation*}
    which is a contradiction since $\nu_k\ne \nu_L$ if $k\ne L$. Hence, \eqref{claim} holds true for some $0<\epsilon \le 1/2$. Consequently, we have 
    \begin{equation}\label{accum}
        B_\epsilon \cap ( M\cup N \cup \mathbb{Z}_{\le0} ) = \{ \mu_{n_0,k}, \mu_{n_0+1,k},\cdots \},
    \end{equation}
    for some $n_0\ge 2$, with $\nu_k-\epsilon <\mu_{n_0,k}< \mu_{n_0+1,k}<\cdots< \nu_k <\nu_k+\epsilon$. If $\nu_k \not \in M$, by considering \eqref{sign}, only elements of $\cb{\Lambda_{n}}_{n\ge n_0-1}$ among $\cb{\Lambda_j}_{j=0}^\infty$ experience sign changes in the interval $(\mu_{n_0,k}, \nu_k+\epsilon)$. We divide the interval into $(\mu_{n,k},\mu_{n+1,k})$ for each $n\ge n_0$ and $[\nu_k,\nu_k+\epsilon)$.

    To be more concrete, by using \eqref{change1}, \eqref{change2} and \eqref{accum}, we have that for each $n\ge n_0$, $\Lambda_n$ is negative on each of intervals $(\mu_{n,k},\mu_{n+1,k})$ and $(\mu_{n+1,k},\mu_{n+2,k})$ but it is positive on each of intervals $[\nu_k,\nu_k+\epsilon)$ and $(\mu_{m,k},\mu_{m+1,k})$ for $m\ge n_0$ with $m\ne n,n+1$. In particular, $\Lambda_{n_0-1}$ is negative on $(\mu_{n_0,k},\mu_{n_0+1,k})$, and positive elsewhere. In other words, if $\nu \in (\mu_{n,k},\mu_{n+1,k})$ for fixed $n\ge n_0$, two elements $\Lambda_{n-1}$ and $\Lambda_{n}$ are negative, while the others are positive. If $\nu \in [\nu_k,\nu_k+\epsilon)$, however, all elements of $\cb{\Lambda_n}_{n\ge n_0-1}$ are positive. Hence, the total number of negative numbers in $\{ \Lambda_j\}_{j=0}^\infty$ increases by $2$ as $\nu$ crosses $\nu_k$ from right to left.
    Moreover, since $\mathcal{N}(\nu)$ is constant on the closed interval $[\nu_k,\nu_k+\epsilon/2]\subset \R \setminus (M\cup \mathbb{Z}_{\le 0})$, we have 
    \begin{equation}\label{accum2}
        \mathcal{N}(\nu_k)=\lim_{\nu\to \nu_{k}+}\mathcal{N}(\nu).
    \end{equation}

    In general, if $\nu_k \in M$, there can be distinct pairs $\cb{(m_j,k_j)}_{j=1}^K $ such that $\nu_k = \mu_{m_j,k_j}$ for all $j=1,\cdots,K$. The same reasoning in the proof of Lemma \ref{lem:5.2} shows that $K$ is finite. Taking $0<\epsilon\le 1/2$ and $n_0-1 > \max\cb{ m_j\,:\, j=1,\cdots,K }$, it is straightforward from \eqref{sign} to see that the only elements of $\cb{\Lambda_{n},\Lambda_{m_j-2},\Lambda_{m_j}\,:\,  n\ge n_0-1,\ j=1,\cdots,K}$,
    among $\cb{\Lambda_j}_{j=0}^\infty$, experience sign changes in the interval $(\mu_{n_0,k}, \nu_k+\epsilon)$. We know that each of $\{\Lambda_n\}_{n\ge n_0-1}$ obeys the same sign pattern as above. Since each of $\cb{\Lambda_{m_j-2},\Lambda_{m_j}\,:\, j=1,\cdots,K}$ follows the sign pattern discussed in \eqref{sign5}, the total number of negative numbers in $\{ \Lambda_j\}_{j=0}^\infty$ increases by $2$ as $\nu$ crosses $\nu_k$ from right to left. The proof is now complete.

\end{proof}

\begin{proof}[Proof of Theorem \ref{thm:main}]
Our proof starts with the case where $\nu>0$. Lemma \ref{lem:5.1} implies that $q_{n,\nu}(1/\nu) = h_n(\nu)>0$ for $\nu>0$, $n\in \mathbb{N}$. Then, it is obvious from Theorem \ref{thm:3.3} that $\Delta_n>0$ for $\nu>0$, $n\in \mathbb{N}$, and thus $\mathcal{N}(\nu)=0$ for $\nu>0$. Hence, by applying Theorem \ref{thm:A} in correspondence with $\sigma_n = \sigma_\nu'(n)$ and $\mathcal{D}_n = \Delta_n$, we conclude that $J_\nu'(x)$ has no complex zeros if $\nu >0$.

In the case where $-1<\nu<0$, we observe that
    \begin{equation*}
        \Lambda_0 = \frac{\nu+2}{2\nu(\nu+1)}< 0.
    \end{equation*}
In view of \eqref{init_behav} and \eqref{sign}, for each $n\ge 1$, $\Lambda_n$ is positive in the vicinity of $0$. Thus, it follows that
    \begin{equation}\label{case1}
        \lim_{\nu\to 0-} \mathcal{N}(\nu)=1.
    \end{equation}
By means of Proposition \ref{prop:4.1} and Remark (iii), we find that $(-1,0)\cap \mathcal{S} = \emptyset$ where $\mathcal{S} = M \cup N \cup \mathbb{Z}_{\le0}$, which leads that $\mathcal{N}(\nu)$ is constant on any closed interval within $(-1,0)$. Hence, \eqref{case1} implies $\mathcal{N}(\nu)=1$ for $-1<\nu <0$ so that $J'_\nu(x)$ has two complex zeros, which are purely imaginary by Proposition \ref{prop:5.1}.

We now shift our attention to the case when $\nu \in (-\infty,-1)\setminus \mathcal{S}$. 
On combining Lemma \ref{lem:5.2}, \ref{lem:5.3} and \ref{lem:5.4} with $\mathcal{N}(\nu)=1$ for $-1<\nu <0$, we have that for each $k=1,2,\cdots$,
\begin{itemize}
    \item $\mathcal{N}(\nu)=k-1$ for $\nu \in (\nu_k,-k)\setminus M$,
    \item $\mathcal{N}(\nu)=k+1$ for $\nu \in (-k-1,\nu_k)\setminus M$.
\end{itemize}

From a technical point of view, the remaining case when $\nu \in \mathcal{S}$ can be divided into $\mathbb{Z}_{\le0}$, $N\setminus M$, and $M\setminus \mathbb{Z}_{\le0}$, which are disjoint from each other. 
If $\nu$ lies in $\mathbb{Z}_{\le 0}$, $J_\nu'(x)$ has only real zeros due to the identities $J_0'(x) = -J_1(x)$ and $J_{-n}'(x) = (-1)^{n} J_{n}'(x)$, $n\in \mathbb{Z}$ (see \cite[\S 2.1 (2), (7)]{Watson}), where $J_1(x)$ has only real zeros according to Theorem \ref{thm:B}.
The case when $\nu\in N\setminus M$ can be dealt using \eqref{accum2}. Thus $J_{\nu_k}'(x)$, $\nu_k \in N \setminus M$, has exactly $(2k-2)$ complex zeros.

On the other hand, the last case when $\nu \in M\setminus \mathbb{Z}_{\le0}$ cannot be directly determined via $\mathcal{N}(\nu)$, because $\mathcal{N}(\nu)$ is undefined for $\nu \in M$. If there is a change in the number of complex zeros at $\nu \in M\setminus \mathbb{Z}_{\le 0}$, it falls into the following scenarios: The first scenario is that some complex zeros become real zeros or vice versa. Specifically, due to the symmetry of zeros, a zero whose multiplicity is no less than twice inevitably arises on the real line in the first scenario, and those non-simple zeros occur only at $j'=0$ or $j'=\pm \nu$. This is because of the fact that, with the possible exception of $j'=0$ or $j'=\pm \nu$, all zeros of $J_\nu'(x)$ are simple, which is easily shown by using the Bessel differential equation.
The second scenario is that some zeros of $J'_{\nu}(x)$ suddenly disappear or come into existence. Notably, \eqref{N-W} shows that the zeros of $J'_{\nu}(x)$ move continuously as $\nu$ varies, possibly except when the right-hand side in \eqref{N-W} diverges, i.e., when $j'=0$ or $j'=\pm \nu$ (see \cite[\S 4]{Lorch2}, \cite[\S 15.6]{Watson} for analogous arguments). In the former case of $j'=0$, $\nu$ should belong to $\mathbb{Z}_{\le 0}$ and the latter case of $j'=\pm \nu$ is only when $\nu\in N$ since $j'$ is a double zero of $J'_\nu(x)$. Thereby, both scenarios possibly happen only when $\nu\in N\cup \mathbb{Z}_{\le 0}$.

Since both scenarios are definitely ruled out at $\nu\in M\setminus (N\cup \mathbb{Z}_{\le 0})$, we conclude that the number of complex zeros for $J_\nu'(x)$ does not change in the small neighborhood of $\nu\in M\setminus (N\cup \mathbb{Z}_{\le 0})$. It remains to figure out the number of complex zeros for $J_\nu'(x)$ when $\nu \in M\cap N$. In fact, the present case fits with the first scenario. To elaborate, we let $\nu = \nu_k \in M\cap N$ for some $k\ge1$. Then, in accordance with the conclusion made above in this paragraph and Lemma \ref{lem:5.4}, it can be shown that, for sufficiently small $\epsilon>0$, $J_\nu'(x)$ has $(2k+2)$ complex zeros if $\nu\in (\nu_k-\epsilon,\nu_k)$, and $(2k-2)$ complex zeros if $\nu \in (\nu_k,\nu_k+\epsilon)$. Whereas, the second scenario does not happen in the case when $\nu \in M\cap N$. Since a zero of $J_\nu'(x)$ is located at the origin only when $\nu\in \mathbb{Z}_{\le0}$, it suffices to consider whether zeros of $J'_{\nu}(x)$ suddenly disappear or come into existence at $x= \pm \nu$ for $\nu \in M\cap N$. We claim that the number of zeros of $J'_\nu(x)$ in a neighbourhood of $x=\nu_k$ is preserved near $\nu=\nu_k$ , i.e., for each $k\in \mathbb{N}$, there exist $\delta>0$ and $\epsilon>0$ such that $J'_\nu(x)$ has exactly two zeros, including multiplicities, in the $\delta$-neighborhood of $x=\nu_k$ for $\nu \in (\nu_k-\epsilon,\nu_k+\epsilon)$. Let $f(x)=J'_{\nu_k}(x)$.
Then, $f(x)$ has a double zero at $x=\nu_k$. Note that $f''(\nu_k)\neq 0$. We take a small $\delta_1>0$ such that $f(x)$ has no other zeros in the disk $|x-\nu_k|\le\delta_1$. In regard to the Taylor expansion of $f(x)$ at $\nu_k$, given by
\begin{equation*}
    f(x)=\frac{J^{(3)}_{\nu_k}(\nu_k)}{2!}(x-\nu_k)^2+\frac{J^{(4)}_{\nu_k}(\nu_k)}{3!}(x-\nu_k)^3+\cdots,
\end{equation*}
we may take small $0<\delta\le \delta_1$ such that 
\begin{equation*}
    |f(x)|> \frac{|J^{(3)}_{\nu_k}(\nu_k)|}{4} \delta^2,
\end{equation*}
on the circle $|x-\nu_k|=\delta$. Now, let us define $ g(x)=J'_{\nu}(x)-J'_{\nu_k}(x)$. Since $\lim_{\nu\to \nu_k}J'_{\nu}(x)=J'_{\nu_k}(x)$ for $x\neq 0$ and
\begin{equation*}
    \max_{|x-\nu_k|=\delta} \big(J'_{\nu}(x)-J'_{\nu_k}(x)\big)
\end{equation*}
is a continuous function of $\nu$, for any $\eta>0$, we can take $\epsilon>0$ such that $|g(x)| <\eta$ on $|x-\nu_k|=\delta$ for $|\nu-\nu_k|<\epsilon$. Taking $\eta = |J^{(3)}_{\nu_k}(\nu_k)|\, \delta^2 /4$,
we have $|g(x)| < |f(x)|$ on $|x-\nu_k|=\delta$ for $|\nu-\nu_k|<\epsilon$. By Rouch\'{e}'s theorem, $f(x)$ and $f(x)+g(x)$ have the same number of zeros in the disk $|x-\nu_k|<\delta$. Since $f(x)$ has only two zeros (up to multiplicity) in the disk, so does $J'_\nu(x)$, which establishes our claim. Thereby, by the fact that $J_\nu'(x)$ has a double zero located at $x=\nu_k$ when $\nu_k\in N$ and the symmetry of the zeros, we deduce that $J_\nu'(x)$ has $(2k-2)$ complex zeros if $\nu = \nu_k$, as if $\nu \in (\nu_k,\nu_k+\epsilon)$.

Therefore, on combining all the results discussed above and Proposition \ref{prop:5.1}, we complete the proof of Theorem \ref{thm:main}.

\end{proof}

\bigskip \noindent
{\bf Acknowledgements.}
The first author is grateful to Mourad E. H. Ismail for helpful comments on an early version of this work.

\end{document}